\documentclass[11pt]{amsart}
\usepackage{amsmath, amsthm, amssymb}
\usepackage{graphicx}
\usepackage[usenames]{color}
\usepackage{srcltx} % JD - to jump between source and preview
\usepackage{verbatim}
\usepackage{mathtools}

\newtheorem{thm}{Theorem}[section]
\newcommand{\bt}{\begin{thm}}
\newcommand{\et}{\end{thm}}

\newtheorem{ex}[thm]{Example}

\newtheorem{cor}[thm]{Corollary}   %remember switch all {coro} to {cor}
\newcommand{\bc}{\begin{cor}}
\newcommand{\ec}{\end{cor}}

\newtheorem{lem}[thm]{Lemma}   %remember to switch all {lem} to {lem}
\newcommand{\bl}{\begin{lem}}
\newcommand{\el}{\end{lem}}

\newtheorem{prop}[thm]{Proposition}
\newcommand{\bp}{\begin{prop}}
\newcommand{\ep}{\end{prop}}

\newtheorem{defn}[thm]{Definition}
\newcommand{\bd}{\begin{defn}}    % This produces an error????    
\newcommand{\ed}{\end{defn}}

\newtheorem{rmrk}[thm]{Remark}   %remember to switch all {rmrk} to {rmrk}

\newcommand{\br}{\begin{rmrk}}
\newcommand{\er}{\end{rmrk}}

\newtheorem{example}[thm]{Example}

\newcommand{\R}{\mathbb R}

\newcommand{\abs}[1]{\left|#1\right|}

\newcommand{\GHto}{\stackrel { \textrm{GH}}{\longrightarrow} }
\newcommand{\Fto}{\stackrel {\mathcal{F}}{\longrightarrow} }

\newcommand{\be}{\begin{equation}}

\newcommand{\ee}{\end{equation}}

\newcommand{\diam}{\operatorname{Diam}}

\newcommand{\set}{\operatorname{set}}

\newcommand{\Lip}{\operatorname{Lip}}

\newcommand{\mass}{{\mathbf M}}

\newcommand{\vol}{\operatorname{Vol}}

\newcommand{\rstr}{\:\mbox{\rule{0.1ex}{1.2ex}\rule{1.1ex}{0.1ex}}\:}

\newcommand{\spt}{\operatorname{spt}}

\newcommand{\lp}{\left (}
\newcommand{\rp}{\right )}

\begin{document}

\title[Sobolev Bounds and Convergence]{Sobolev Bounds and Convergence of Riemannian Manifolds}

\author{Brian Allen}
\address{USMA}
\email{brianallenmath@gmail.com}

\author{Edward Bryden}
\address{Stony Brook}
\email{etbryden@gmail.com}

%\author{C. Sormani}
%\thanks{C. Sormani was partially supported by NSF DMS 1006059.}
%\address{CUNY Graduate Center and Lehman College}
%\email{sormanic@gmail.com}

%\date{February 2017}

%\keywords{}

%49Q15 (Geometric measure and integration theory, integral and normal currents)

%\subjclass[2000]{49Q15}

\begin{abstract}
We consider sequences of compact Riemannian manifolds with uniform Sobolev bounds on their metric tensors, and prove that their distance functions are uniformly bounded in the H\"{o}lder sense. This is done by establishing a general trace inequality on Riemannian manifolds which is an interesting result on its own. We provide examples demonstrating how each of our hypotheses are necessary.  In the Appendix by the first author with Christina Sormani, we prove that sequences of compact integral current spaces without boundary (including Riemannian manifolds) that have uniform H\"{o}lder bounds on their distance functions have subsequences converging in the Gromov--Hausdorff (GH) sense.  If in addition they have a uniform upper bound on mass (volume) then they converge in the Sormani--Wenger Intrinsic Flat (SWIF) sense to a limit whose metric completion is the GH limit.  We provide an example of a sequence developing a cusp demonstrating why the SWIF and GH limits may not agree.
\end{abstract}

\maketitle

\section{Introduction}\label{sec: Intro}

When studying stability problems where metric convergence notions such as uniform, Gromov-Hausdorff (GH) or Sormani-Wenger Intrinsic Flat (SWIF) convergence are appropriate, it is common to obtain Sobolev bounds in a special coordinate system first. Then the question which naturally follows is how do we use the acquired Sobolev bounds to obtain uniform, GH and/or SWIF convergence? The goal of this paper is to give an answer to this question through compactness results involving Sobolev bounds on a sequence of metrics. 

Sobolev bounds on metric tensors have been applied to solve a special case of Gromov's Conjecture on the Stability of the Scalar Torus Theorem \cite{Gromov-Plateau, Sormani-scalar}. In particular, Sobolev bounds were used by the first named author, Hernandez-Vazquez, Parise, Payne, and Wang in \cite{AHPPW1} to show uniform, GH and SWIF convergence to a flat torus, demonstrating the stability of the scalar torus rigidity in the warped product case. This provides evidence of the potential application of Sobolev bounds to showing stability of the scalar torus rigidity theorem for the general case to address the conjecture of Gromov and Sormani, which is stated in \cite{AHPPW1}.

Sobolev bounds on metric tensors have also arisen when working to prove special cases of Lee and Sormani's Conjecture on the Stability of the Schoen-Yau Positive Mass Theorem \cite{Schoen-Yau-positive-mass, LeeSormani1}. For instance, the second named author has shown $W^{1,p}$, $1 \le p < 2$, stability of the positive mass theorem (PMT) in the case of axisymmetric, asymptotically flat manifolds with additional technical assumptions stated precisely in \cite{Bryden-Stability}. Ultimately, the goal of this work is to show SWIF convergence in the axisymmetric case in order to address the conjecture of Lee and Sormani \cite{LeeSormani1} on stability of the PMT. Similarly, the first named author has shown $L^2$ stability of the PMT using Inverse Mean Curvature Flow (IMCF) \cite{Allen-Inverse} and has recently shown $W^{1,2}$ stability of the PMT using IMCF \cite{Allen-InverseSobolev}. The authors expect the results of this paper will help to obtain SWIF convergence using the previously established Sobolev convergence.

Our main theorem gives conditions on a sequence of Riemannian manifolds which guarantee uniform H\"{o}lder bounds on the corresponding distance functions which allow us to prove compactness. The bounds required to obtain the uniform, GH, and SWIF convergence of Theorem \ref{HLS-thm} are an improvement of a theorem in the appendix of \cite{HLS}. No expertise on GH or SWIF convergence is needed to read the rest of the paper.

\begin{thm}\label{SobolevBoundsAndHolderCompactness}
Let $M^n$ be compact, possibly with boundary, $M_j=(M,g_j)$ a sequence of Riemmanian manifolds, and $M_0=(M,g_0)$ a background Riemannian manifold with metric $d_0$. If
\begin{align}
\abs{\abs{|g_j|_{g_0}}}_{W_{g_0}^{n-1,p}(M)} \le C, p > 1\label{SobolevUpperBound1}
\end{align}
 and
\begin{align}
g_j(v,v) \ge c g_0(v,v)\: \forall q \in M, v \in T_qM,\label{MetricTensorLowerBound1}
\end{align}
then there exists a subsequence which converges in the uniform and GH sense to a length space $M_{\infty}=(M,d_{\infty})$ so that
\begin{align}
\sqrt{c} d_0(p,q) \le d_{\infty}(p,q) \le C^{\frac{1}{p}} d_0(p,q)^{\frac{p-1}{p}},\label{LimitingMetricBounds}
\end{align}
where $C$ depends on $n, p$, the geometry of $M_0$, and the constant in \eqref{SobolevUpperBound1}.

In addition, the subsequence can be chosen to SWIF converge to $M_\infty'=(M', d_\infty, T_\infty)$ where the metric completion of $M'$ is $M$ and $T_{\infty}$ is an integral current.
\end{thm}

For comparison purposes it is useful to understand what notions of convergence are known to imply GH and SWIF convergence. It was shown by Sormani and Wenger that Gromov Lipschitz convergence (see \cite{Gromov-metric} for the definition) implies SWIF convergence \cite{SW-JDG} and Gromov showed that Lipschitz convergence implies GH convergence \cite{Gromov-metric}. In general, GH and SWIF convergence do not have the same limits: cusp points disappear under SWIF convergence as do entire regions that collapse. It was shown by Sormani and Wenger \cite{SW-CVPDE}, and Matveev and Portegies \cite{MP-PropIF} that GH=SWIF for sequences of noncollapsing Riemannian manifolds with lower bounds on Ricci curvature, and Huang, Lee, and Sormani for sequences of Riemannian manifolds with biLipschtiz bounds on their distance functions \cite{HLS}, and by Perales for special sequences of manifolds with boundary \cite{RP-Boundary}. The first author and Sormani \cite{Allen-Sormani-I} gave conditions which guarantee that the $L^2$ convergence of a sequence of warped products will agree with the uniform, GH and SWIF convergence. In particular, one should note that Example 3.4 of \cite{Allen-Sormani-I} shows that $L^p$ convergence or $W^{k,p}$ convergence is not enough to imply GH or SWIF convergence in general. 

One should also contrast these results to the $C^{\alpha}$ compactness theorem of Anderson and Cheeger \cite{Anderson-Cheeger} where it is shown that the space of compact Riemannian manifolds with lower bounds on Ricci curvature and injectivity radius, and an upper bound on volume, are precompact with respect to the $C^{\alpha}$ notion of distance between Riemannian manifolds. Anderson and Cheeger use the geometric bounds in order to show that a finite collection of harmonic coordinate charts must exist which is well controlled in the $C^{\alpha}$ sense. In the present work we aim to make weaker assumptions on the sequence of Riemannian manifolds at the cost of the weaker notions of convergence like uniform, GH and SWIF. It should be noted though that if one chooses $p > \frac{n}{n-1}$ in Theorem \ref{SobolevBoundsAndHolderCompactness} then by the Sobolev embedding theorem one can infer $C^{\alpha}$ bounds on $g_j$ in terms of the background metric $g_0$ and hence by Arzela-Ascoli we would obtain precompactness in the $C^{\alpha}$ sense. Hence the main importance of the above theorem is when $p \le \frac{n}{n-1}$.

In order to prove our main theorem we prove the following trace inequality which shows how assumption \eqref{SobolevUpperBound1} of Theorem \ref{SobolevBoundsAndHolderCompactness} controls integrals of $|g_j|_{g_0}$ along geodesics with respect to the background metric $g_0$. The following theorem is a generalization of results of the second named author in \cite{Bryden-Stability} where bounds of this type were used to estimate the distances between points in axisymmetric initial data with small ADM mass which satisfy an additional technical hypothesis.

\begin{thm}\label{SobolevTraceTrick}
	Let $(K,g)$ be a smooth compact Riemannian $n$-dimensional manifold, possibly with boundary, and let $p \ge 1$. There exists a constant, denoted $C=C(g,K)$, such that
	\begin{equation}
		\left(\int_{\gamma}\abs{f}^pdt\right)^{\frac{1}{p}}\le C\abs{\abs{f}}_{W^{n-1,p}_{g}(K)},
	\end{equation}
	for length minimizing geodesics $\gamma$ with respect to $g$. 
\end{thm}
\begin{rmrk}
Note that in the case where $p > \frac{n}{n-1}$ we can easily apply Sobolev embedding theorems to deduce H\"{o}lder control and so the nontrivial case of Theorem \ref{SobolevTraceTrick} is when $1 \le p \le \frac{n}{n-1}$.
\end{rmrk}

To complete the proof of the main theorem, we apply the following theorem proven in the appendix by the first named author with Christina Sormani.

\begin{thm}\label{HLS-thm}
Let $M$ be a fixed compact, $C^0$ manifold with piecewise continuous metric tensor $g_0$ and $g_j$ a sequence of piecewise continuous Riemannian metrics on $M$, defining distances $d_j$ such that
\be\label{d_j1}
\frac{1}{\lambda}d_0(p,q) \le  d_j(p,q) \le \lambda d_0(p,q)^{\alpha},
\ee
for some $\lambda \ge 1$ where $\alpha\in (0,1]$ then there exists a subsequence of $(M, g_j)$
that converges in the uniform and Gromov-Hausdorff sense. In addition, if there is a uniform upper bound on volume
\be \label{HLS-thm-vol1}
\vol(M, g_j) \le V_0.
\ee 
the subsequence converges in the intrinsic flat sense to the same space with a possibly different current structure than the sequence(up to possibly
taking a metric completion of the intrinsic flat limit to get
the Gromov-Hausdorff and uniform limit). 
\end{thm}

This theorem is an improvement of a theorem which appeared in the appendix of the work by Huang, Lee and Sormani \cite{HLS} where uniform Lipschitz bounds on a sequence of distance functions is assumed in order to imply that a subsequence converges in the uniform, GH and SWIF sense. We replace the Lipschitz bounds by H\"{o}lder bounds in order to also show that a subsequence converges in the uniform, GH and SWIF sense. One difference in the conclusion is that the GH and SWIF limits could disagree at cusp points as illuminated by Example \ref{CuspEx}. 

Example~\ref{notpyramids} provides a sequence of piecewise flat manifolds which satisfy the hypothesis \eqref{d_j1} of  Theorem \ref{HLS-thm} but fail hypotheses \eqref{HLS-thm-vol1} because their volumes diverge to infinity.
We do not require the additional hypothesis on the volume in Theorem \ref{SobolevBoundsAndHolderCompactness} because we prove in
Section \ref{sec: MainThmProof} that hypothesis \eqref{SobolevUpperBound1} implies an upper bound on volume holds.

Example \ref{NonUniform} gives a sequence of warped products where a uniform Lipschitz bound does not exist but a uniform H\"{o}lder bound does exist and hence one can show that a subsequence converges to the cylinder in the uniform, GH and SWIF sense. This is an important example since it illustrates the necessity of having an analogous theorem to the one found in the appendix of \cite{HLS} in terms of H\"{o}lder bounds.

In Section \ref{sec: Background}, we review the definitions of uniform, GH and SWIF convergence of Riemannian manifolds. We stress that the reader does not need to be an expert in GH or SWIF convergence and could choose to appreciate each result in terms of uniform convergence if desired.

In Section \ref{sec: Examples}, we give examples which illustrate the necessity of the assumptions of our main theorems.

In Section \ref{sec:Trace}, we show how a trace theorem can be used to imply $L^p$ bounds on the metric tensor along geodesics of a background metric from Sobolev bounds. 

In Section \ref{sec: MainThmProof}, we give the proof of Theorem \ref{SobolevBoundsAndHolderCompactness} as well as prove Theorem \ref{HolderCompactness} which is a similar theorem with the Sobolev bound replaced by an $L^p$ bound of the metric along geodesics of the background metric.

In Section \ref{sec: App}, the first author and Christina Sormani give a proof of Theorem \ref{HLS-thm}. One should note that the uniform convergence of Theorem \ref{HLS-thm} follows from the Arzela-Ascoli theorem and requires no knowledge of GH and SWIF convergence. The majority of the proof is to verify technical results involving the support and settled completion of the integral currents which define the SWIF convergence.

\vspace{.2in}
\noindent{\bf Acknowledgements:} The author's would like to thank Marcus Khuri, Dan Lee, and Christina Sormani for their constant support and encouragement. In addition, we would like to thank Jorge Basilio, Lisandra Hernandez-Vazquez, Demetre Kazaras, Sajjad Lakzian, and Raquel Perales for useful discussions. The first named author would also like to thank Theodora Bourni, Carolyn Gordon, Mat Langford, Marco Mendez, Bjoern Muetzel, Andre Neves, Xiaochong Rong, Thomas Yu and Gang Zhou for recent invitations to speak in their respective seminars and conferences. Christina Sormani is supported by NSF grant DMS-1612049 which has also supported the authors travel.

\section{Background} \label{sec: Background}

In this section we review the definitions of uniform, GH and SWIF convergence of Riemannian manifolds. Throughout we will give references so that the interested reader can dig deeper into the details of these definitions of convergence if desired. We note that the reader who is not familiar with GH and SWIF convergence could choose to solely digest the definition of uniform convergence of metric spaces and understand each result of the paper in terms of this convergence alone.

\subsection{Uniform Convergence}\label{subsec: Uniform Def}

We define the uniform distance between the metric spaces $(X,d_1)$, $(X,d_2)$  to be
\begin{align}
d_{unif}(d_1,d_2) = \sup_{x,y\in X} |d_1(x,y) - d_2(x,y)|.
\end{align}
If one thinks of the metrics as functions, $d_i: X\times X \rightarrow \R$, then the uniform distance $d_{unif}(d_1,d_2)$ is equivalent to the $C^0$ distance between functions. We say that a sequence of metrics spaces $(X, d_i)$ converges to the metric space $(X,d_{\infty})$ if $d_{unif}(d_i, d_{\infty}) \rightarrow 0$ as $i \rightarrow \infty$.

One limitation of uniform convergence is that it requires the metric spaces to have the same topology. In the current paper this is not a problem since we are assuming that the sequence of Riemannian manifolds are all defined on the same background manifold. See the text of Burago, Burago, and Ivanov  \cite{BBI} for more information on uniform convergence.

\subsection{Gromov-Hausdorff Convergence}\label{subsec: GH Def}

Gromov-Hausdorff (GH) convergence was introduced by Gromov in \cite{Gromov-metric}
and which is discussed in the text of Burago, Burago, and Ivanov \cite{BBI}. Given two metric spaces we can define the GH distance
between them, which is more general than uniform convergence since it doesn't require the metrics spaces to have the same topology. It is the intrinsic version of the Hausdorff distance between
sets in a common metric space $Z$ which is defined as
\begin{equation} 
d_H^Z(U_1, U_2) = \inf\{ r \, : \, U_1\subset B_r(U_2) \textrm{ and } U_2\subset B_r(U_1)\},
\end{equation} 
where $B_r(U)=\{x \in Z: \, \exists y \in U \, s.t.\, d_Z(x,y)<r\}$.
Given a pair of compact metric spaces, $(X_i,d_i)$ we can use distance preserving maps to embed both metric spaces in a common, compact metric space $Z$. A distance preserving map is defined by
\begin{equation}
\varphi_i: X_i \to Z \textrm{ such that } d_Z(\varphi_i(p), \varphi_i(q)) = d_i(p,q) \,\, \forall p,q \in X_i
\end{equation}  
where it is important to note that we are requiring a metric isometry here which is stronger than a Riemannian isometry.

The Gromov-Hausdorff distance between two compact metric spaces, $(X_i,d_i)$,
is then defined to be
\begin{equation}
d_{GH}((X_1,d_1),(X_2,d_2))=\inf\{ d_H^Z(\varphi_1(X_1), \varphi_2(X_2)) \, : \, \varphi_i: X_i\to Z\}
\end{equation}
where the infimum is taken over all compact metric spaces $Z$ and all
distance preserving maps, $\varphi_i: X_i \to Z$. We say a sequence of metric spaces, $X_i$, converges to a metric space, $X_{\infty}$, if $d_{GH}(X_i,X_{\infty}) \rightarrow 0$.

\subsection{Sormani-Wenger Intrinsic Flat Convergence} \label{subsec: SWIF}

Gromov-Hausdorff distance between metric spaces is an extremely powerful and useful notion of distance but some stability problems have been shown to be false under GH convergence \cite{Sormani-scalar}. To this end we now define another notion of convergence, introduced by Sormani and Wenger in \cite{SW-JDG}, which is defined on integral currents spaces.

The idea is to build an intrisic version of the flat distance on $\R^n$ of Federer and Fleming \cite{FF} for any metric space. Using the construction of currents on metric spaces by Ambrosio and Kirchheim \cite{AK} one can define the flat distance for currents $T_1, T_2$ on a metric space $Z$ as follows
\begin{align}
d_F^Z(T_1,T_2) = \inf\{\mass^n(A)+\mass^{n+1}(B): A + \partial B = T_1 - T_2\}.
\end{align}

Sormani and Wenger \cite{SW-JDG} then use this notion of flat convergence to define the intrinsic notion of convergence for integral currents spaces  $M_1=(X_1,d_1,T_1)$ and $M_2=(X_2,d_2,T_2)$ as
\begin{align}
d_{\mathcal{F}}(M_1,M_2) = \inf\{d_F^Z(\varphi_{1\#}T_1,\varphi_{2\#}T_2): \varphi_j: X_j \rightarrow Z\}
\end{align} 
where the infimum is taken over all complete metric spaces $Z$ and all metric isometric embeddings $\varphi_j: X_j \rightarrow Z$. See \cite{SW-JDG, Sormani-scalar} for the definition of integral current spaces. Note that compact, oriented Riemannian manifolds are integral current spaces.

We say a sequence of integral current spaces, $X_i$, converge to the integral current space, $X_{\infty}$, if $d_{\mathcal{F}}(X_i,X_{\infty}) \rightarrow 0$.

For the limiting integral current space to be well defined under a SWIF limit it is important for points to have positive lower density. See the discussion that follows equation \eqref{setDef} of the definition of the set of positive density of $T$ in subsection \ref{subsec: Currents Background}. In Example \ref{CuspEx} we see that when a sequence converges in the GH and SWIF sense they may not agree on the limit space due to disappearing points. Despite this fact we see that $Cl(X')=X$ in Theorem \ref{HLS-thm} for sequences with H\"{o}lder bounds even though this is not necessarily true in general (See Example 3.10 of Lakzian and Sormani \cite{Lakzian-Sormani-1}).

\section{Examples} \label{sec: Examples}

In this section we discuss examples which illustrate the necessity of the hypotheses of the main theorems. In most of the examples we consider warped products with warping factors,
\be
f:[a,b]\to \mathbb{R}^+,
\ee
$a,b \in \R$, where we define the warped product manifolds,
\be \label{MorN}
M=[a,b]\times_f S^1,
\ee
with coordinates $(r,\theta) \in M$ and warped product metric,
\be \label{eqn-g}
g= dr^2 + f^2(r) d\theta .
\ee

\subsection{Lack of Uniform Lipschitz Bound}\label{subsec: NoLipschitzBound}

Our fist example shows a sequence of Riemannian manifolds which converge in the uniform, GH and SWIF sense to the same limit manifold but which does not satisfy a uniform Lipschitz bound. Instead we will show that a uniform H\"{o}lder bound does exist and hence Theorem \ref{HLS-thm} will apply to give a converging subsequence.

\begin{ex}\label{NonUniform} 
We construct a sequence of smooth functions
$f_j:[-1,1]\to [1,\infty)$ which converges pointwise almost everywhere to $f_\infty=1$.
Define
  \be
 f_j(r)=
 \begin{cases}
  h_j(jr ) & r\in [-\frac{1}{j}, \frac{1}{j}] 
 \\ 1 & \textrm{ elsewhere }
 \end{cases}
\ee
where $h_j$ is a smooth even function such that 
$h_j(-1)=1$ with $h_j'(-1)=0$, 
increasing up to $h_j(\frac{-1}{2})=j^{\eta}+1$, $\eta \in(0,1)$ with $h_j'(\frac{-1}{2})=0$, constant on $[\frac{-1}{2},\frac{1}{2}]$ with $h_j'(\frac{1}{2})=0$, and then
decreasing back down to $h_j(1)=1$ with $h_j'(1)=0$. This defines the Riemannian warped product $M_j = [a,b]\times_{f_j}S^1$.
Consider also $M_\infty$ defined as above with $f_\infty(r)=1$, $r \in [-1,1]$.

We will show that, despite the fact that no uniform Lipschitz bound exists, a uniform H\"{o}lder bound does exist,
\begin{align}
d_j(p,q) \le C d_{\infty}(p,q)^{\beta},
\end{align}
where $\beta = 1 - \eta$, as well as a uniform upper bound on volume, and hence  $M_j \rightarrow M_{\infty}$ in uniform, GH  and SWIF sense.
 \end{ex}
 \begin{proof}
Now if we consider $p_j = \left ( -\frac{1}{8j^2},0\right)$, $q_j =  \left ( \frac{1}{8j^2}, \frac{1}{8j^2}\right)$ and $C_j(t) = \left(\frac{t}{8j^2},\frac{1}{16 j^2}(t+1) \right)$ for $t \in [-1,1]$ then we know that this is a geodesic since $f_j(r)$ is identically constant for $r \in [-\frac{1}{4j},\frac{1}{4j}]$. Now we calculate 
\begin{align}
L_{g_j}(C_j)&= \int_{-1}^1 \sqrt{\left(\frac{1}{8j^2}\right)^2 + (j^{\eta}+1)^2 \frac{1}{16^2 j^4}} dt 
\\&= 2 \sqrt{\left(\frac{1}{8j^2}\right)^2 +  \frac{1}{16^2} \frac{j^{2 \eta} + 2 j^{\eta} + 1}{j^4}} 
\\&=\frac{1}{8} \frac{\sqrt{j^{2 \eta}+2j^{\eta}+5}}{j^2}
\end{align}
which must be the minimum length of any curve which stays in the region where $r \in [-\frac{1}{4j},\frac{1}{4j}]$. We would like to show that this curve realizes the distance between $p_j$ and $q_j$ for $j$ chosen large enough. Consider another curve $\gamma_j = (r_j(t), \theta_j(t))$, with end points $p_j$ and $q_j$, which leaves the $r \in [-\frac{1}{4j},\frac{1}{4j}]$ region and then re-enters which has length
\begin{align}
L_{g_j}(\gamma_j)&= \int_{-1}^1 \sqrt{r'_j(t)^2 + (f_j\theta'_j(t))^2} dt \ge \int_{-1}^1|r'_j|dt \ge \frac{1}{4j}.
\end{align}
This shows for $j$ large enough that $L_{g_j}(C_j) < L_{g_j}(\gamma_j)$ and so $d_{g_j}(p_j,q_j) = L_{g_j}(C_j)$.
Now we compute 
\begin{align}
d_{g_{\infty}}(p_j,q_j) = 2\sqrt{\frac{1}{8^2j^4}+ \frac{1}{16^2j^4}} = 2\sqrt{\frac{1}{8^2}+ \frac{1}{16^2}}j^{-2} = \frac{\sqrt{5}}{8} j^{-2}
\end{align}
 which shows that,
\begin{align}
 \frac{d_{g_j}(p_j,q_j)}{d_{g_{\infty}}(p_j,q_j)}&= \frac{1}{\sqrt{5}}\sqrt{j^{2 \eta}+2j^{\eta}+5}\rightarrow \infty,
\end{align}
and hence no uniform Lipschitz bound exists.  Instead we  notice that 
\begin{align}
\frac{d_{g_j}(p_j,q_j)}{d_{g_{\infty}}(p_j,q_j)^{1 - \frac{\eta}{2}}}&= \frac{1}{\sqrt{5}}\frac{\sqrt{j^{2 \eta}+2j^{\eta}+5}}{j^{\eta}}\le C.
\end{align}
Now our goal is to show that a uniform H\"{o}lder bound holds for comparing distances between any point and hence Theorem \ref{HolderCompactness} applies to obtain the desired convergence. 

To this end, consider $p,q \in M$, $p=(r_1,\theta_1)$, $q=(r_2,\theta_2)$ and let $\gamma$ be a straight line curve in coordinates connecting $p$ to $q$. Then if either $r_1\not \in (-\frac{1}{j},\frac{1}{j})$, $r_2 \not \in (-\frac{1}{j},\frac{1}{j})$, or both then we can estimate the $d_j$ distance by a curve which travels solely in the $r$ direction followed by a curve which travels solely in the $\theta$ direction, such that the part of the curve which travels in the $\theta$ direction does not occur within the region where $r \in (\frac{-1}{j},\frac{1}{j})$, to notice that,
\begin{align}
d_j(p,q) \le |r_1-r_2| + |\theta_1-\theta_2|.
\end{align}
Then we can use the curve $\gamma$ to calculate the $d_{\infty}$ distance in order to show that
\begin{align}
\frac{d_j(p,q)}{d_{\infty}(p,q)^{1-\eta}} \le \frac{|r_1-r_2| + |\theta_1-\theta_2|}{(|r_1-r_2|^2 + |\theta_1-\theta_2|^2)^{\frac{1-\eta}{2}}} \le C.
\end{align}
Now we consider $p=(r_1,\theta_1)$, $q=(r_2,\theta_2)$ in the case where $r_1, r_2 \in (-\frac{1}{j},\frac{1}{j})$ in which case the curve $\gamma$ is contained in the region where $r \in (-\frac{1}{j},\frac{1}{j})$. In this case we estimate
\begin{align}
d_j(p,q) &\le L_{g_j}(\gamma) 
\\&\le \sqrt{|r_1-r_2|^2+(1 + j^{\eta})^2 |\theta_2-\theta_1|^2}
\\&\le |r_1-r_2|+(1 + j^{\eta}) |\theta_2-\theta_1|.
\end{align}

Now we use this estimate to calculate
\begin{align}
&\frac{d_j(p,q)}{d_{\infty}(p,q)^{1- \eta}}\le \frac{|r_1-r_2|+(1 + j^{\eta}) |\theta_2-\theta_1|}{\lp|r_1-r_2|^2+|\theta_2-\theta_1|^2\rp^{\frac{1}{2}-\frac{\eta}{2}}}
\\&\le \frac{|r_1-r_2|+(1 + j^{\eta}) |\theta_2-\theta_1|}{|r_1-r_2|^{1 - \eta}} 
\\&= \frac{|r_1-r_2|^{1+\eta}+(1 + j^{\eta}) |r_1-r_2|^{\eta}|\theta_2-\theta_1|}{|r_1-r_2|} \label{Before}
\\&\le \frac{|r_1-r_2|^{1+\eta}+2\frac{(1 + j^{\eta})}{j^{\eta}} |\theta_2-\theta_1|}{|r_1-r_2|} \le C \label{After}
\end{align}
where we used that $|r_1-r_2| \le \frac{2}{j}$ to go from \eqref{Before} to \eqref{After}.

Hence we have shown a uniform H\"{o}lder bound from above. Notice that a uniform Lipschitz bound from below follows immediately from the definition of $g_j$ and we can find a uniform volume bound
\begin{align}
Vol(M_j) \le \frac{4\pi(j^{\eta}+1)}{j}+4\pi\left(1 - \frac{1}{j}\right) \le 16 \pi.
\end{align}
Hence Theorem \ref{HLS-thm} applies to show that a subsequence converges in the uniform, GH, and SWIF sense, as desired.

 To show specifically that the subsequence converges to $M_{\infty}$ we will show pointwise convergence of distances. Note that if $p=(r_1,\theta_1)$, $q=(r_2,\theta_2)$, and $r_1,r_2 \in [-1,0)$ or $r_1,r_2 \in (0,1]$ then the pointwise convergence follows by choosing $j$ large enough so that the straight line curve in coordinates joining $p$ to $q$ avoids the region where $r \in (\frac{-1}{j},\frac{1}{j})$. 
 
 In the case where $r_1 \in [-1,0]$ and $r_2\in(0,1]$, or $r_1 \in [-1,0)$ and $r_2\in [0,1]$ then we can consider a sequence of approximating curves, $C_j$, which travel along a straight line until the region where $r \in (\frac{-1}{j},\frac{1}{j})$ and then travel solely in the $r$ direction. We can use these approximating curves to estimate
 \begin{align}
 d_j(p,q) \le L_{g_j}(C_j) \le \sqrt{\lp|r_2-r_1| - \frac{c}{j}\rp^2 + |\theta_2-\theta_1|^2} + \frac{c}{j},
 \end{align}
 for some constant $c>0$, which when combined with the lower bound on distance implied by the fact that $g_j \ge g_{\infty}$ implies pointwise convergence of distances in this case. 
 
 In the case where $r_1,r_2 = 0$ we can build a sequence of approximating curves $C_j'$ which move $\frac{2}{j}$ in the $r$ direction, followed by traveling solely in the theta direction outside the region where $r \in \left(\frac{-1}{j},\frac{1}{j} \right)$, followed by a curve which moves $\frac{2}{j}$ in the $r$ direction. We can use these approximating curves to estimate 
  \begin{align}
 d_j(p,q) \le L_{g_j}(C_j') \le |\theta_2-\theta_1| + \frac{4}{j},
 \end{align}
 which when combined with the lower bound on distance implied by the fact that $g_j \ge g_{\infty}$ implies pointwise convergence of distances in this case. Hence we have shown pointwise converge of distances which when combined with Theorem \ref{HLS-thm} shows that every subsequence of $M_j$ converges to $M_{\infty}$ with respect to uniform, GH and SWIF convergence and hence the sequence must converge as well.

\end{proof}

\subsection{Lack of Uniform H\"{o}lder Bound}\label{subsec: NonHolderBound}

In the next example a uniform upper H\"{o}lder bound does not exist and so we see that the sequence converges pointwise to a limiting metric space which is totally disconnected but does not converge uniformly.

\begin{ex}\label{NonHolderBoundAboveEx}
Consider the sequence of metric spaces $([0,1],d_j)$ with distance functions
\begin{align}
d_j(x,y) = |x-y|^{1/j},
\end{align}
then we see that $d_j$ converges pointwise as a function to the metric $d_{\infty}$ defined as
\begin{align}
d_{\infty}(x,y)=
\begin{cases}
1 & \text{ if } x \not = y
\\0 & \text{ if } x = y
\end{cases}
\end{align}
but $d_j$ does not converge uniformly to $d_{\infty}$ and $([0,1],d_{\infty})$ is in fact totally disconnected.
\end{ex}

\subsection{Lack of Uniform Inverse H\"{o}lder Bound}\label{NoInverseHolderEx}

Notice that the lower H\"{o}lder bound of Theorem \ref{HLS-thm} is equivalent to saying that the inverse identity map is H\"{o}lder continuous. In the next example we see that if we do not assume a uniform inverse H\"{o}lder bound then it is possible for points to be identified in the limit, thus changing the topology of the space on which the limiting metric is defined. This example first appeared as Example 3.4 in the work of Lakzian and Sormani \cite{Lakzian-Sormani-1}.

\begin{ex}
Consider the sequence of warped product metrics $g_i$ on the sphere, $S^2$, with warping functions which are carefully defined in Example 3.4 of \cite{Lakzian-Sormani-1}, which converge smoothly away from the equator such that the equator pinches to zero. The GH and SWIF limits of this sequence are both two spheres joined at a point which is the identification of the equator. Here we notice that because of the lack of a uniform inverse H\"{o}lder bound on the sphere the sequence is able to identify the equator as one point in the limit since the distances between all points on the equator are converging to $0$. 
\end{ex}

\subsection{Cusps and Disappearing Points}\label{subsec: CuspsEx}
This next example first appeared as Example A.9. in Section 6 of the work of Sormani and Wenger \cite{SW-JDG}. An explicit construction was later given as Example 3.3 in the work of Sormani and Lakzian \cite{Lakzian-Sormani-1}. This example illustrates how the GH and SWIF limits guaranteed by Theorem \ref{SobolevBoundsAndHolderCompactness}, Theorem \ref{HLS-thm} and Theorem \ref{HolderCompactness} can disagree due to disappearing points. 

\begin{ex}\label{CuspEx}
Consider the warped product $g_j$ with warping function
\begin{align}
f_j(r) = (1/j) \sin(r) +(1-1/j) f(r) \text{ for } r \in [0,\pi]
\end{align}
where $f(r)$ is a smooth function defined by
\begin{align}
f(r)=
\begin{cases}
\sin(r) & \text{ for } r \in [0,\pi/2],
\\\frac{4}{\pi^2}(r-\pi)^2 & \text{ for } r \in [3 \pi/4,\pi].
\end{cases}
\end{align}
This is a sequence of metrics on the sphere $S^2$ which converges smoothly away from the point singularity $p_0 = (\pi,0)$. Then we find
\begin{align}
M_j &\GHto (S^2,d_{\infty})
\\M_j &\Fto (S^2 \setminus \{p_0\},d_{\infty})),
\end{align}
where $d_{\infty}$ is the spherical metric with a cusp tip.
This discrepency occurs due to a cusp singularity forming at the point $p_0$ which has density $0$ and hence is not included in the settled completion of the limit in the case of SWIF convergence.
\end{ex}

In Example \ref{CuspEx} we saw a cusp forming at one point and hence this point was not included in the SWIF limit space. One can now imagine placing as many cusp points as one likes to create a similar example illustrating the fact that under Theorem \ref{SobolevBoundsAndHolderCompactness} and Theorem \ref{HolderCompactness} the GH and SWIF limits can disagree.

\subsection{Cones and Lack of Smooth Convergence} \label{subsec: ConeEx}

This next example first appeared as Example A.8. in Section 6 of the work of Sormani and Wenger \cite{SW-JDG}. An explicit construction was later given as Example 3.2 in the work of Sormani and Lakzian \cite{Lakzian-Sormani-1}. This example shows that even when the GH and SWIF limits agree we should not expect smooth convergence in the conclusion of Theorem \ref{SobolevBoundsAndHolderCompactness} and Theorem \ref{HolderCompactness}.

\begin{ex}
Consider the warped product metrics $g_j$ with warping functions
\begin{align}
f_j(r) = (1/j) \sin(r) + (1 - 1/j) f(r) \text{ for } r \in [0,\pi]
\end{align}
where we define $f$ to be a smooth function where
\begin{align}
f(r)=
\begin{cases}
\sin(r) & \text{ for } r \in [0,\pi/2]
\\ -\frac{2}{\pi}(r - \pi) & \text{ for } r \in [3\pi/4,\pi].
\end{cases}
\end{align}
$M_j$ is then a sequence of metrics on the sphere $S^2$ which converges smoothly away from the point $p_0 = (\pi,0)$. Then we find that
\begin{align}
M_j &\GHto (S^2,d_{\infty}),
\\M_j &\Fto (S^2,d_{\infty}),
\end{align}
where $d_{\infty}$ is the spherical metric with a cone tip. The cone tip demonstrates that $M_j$ does not converge smoothly to $M_{\infty}$.
\end{ex}

\subsection{Necessity of Metric Lower Bounds} \label{subsec: MetricLowerBoundsEx}
Example \ref{Cinched-Torus} first appeared as Example 3.4 of the work of the first named author and Sormani \cite{Allen-Sormani-I} which demonstrates that without a metric lower bound like \eqref{MetricTensorLowerBound1} which holds pointwise everywhere one should not expect the lower bound on the limiting metric as in \eqref{LimitingMetricBounds}.

\begin{ex} \label{Cinched-Torus}  
 Consider the sequence of smooth functions $f_j(r):[-\pi,\pi]\to [1,2]$ 
 \be
 f_j(r)=
 \begin{cases}
 1 & r\in[-\pi,- 1/j]
 \\  h(jr) & r\in[- 1/j, 1/j]
 \\ 1 &r\in [1/j, \pi]
 \end{cases}
\ee
where $h$ is a smooth even function such that 
$h(-1)=1$ with $h'(-1)=0$, 
decreasing to $h(0)=h_0\in (0,1]$ and then
increasing back up to $h(1)=1$, $h'(1)=0$. 
Note that this defines a sequence of smooth Riemannian metrics, $g_j$,
as in (\ref{eqn-g}), with corresponding distances $d_j$ on the manifolds, 
\be
M_j= [-\pi,\pi]\times_{f_j} \Sigma \textrm{ or } N_j={\mathbb{S}}^1\times_{f_j} \Sigma
\ee 
for any fixed Riemannian manifold $\Sigma$.   Consider also
$M_\infty$ and $N_\infty$ defined as above with $f_\infty(r)=1 \quad \forall r$.

Despite the fact that
\be
f_j \to f_\infty \textrm{ in } L^p \quad 
\ee
we do not have $M_j$ converging to $M_\infty$ nor $N_j$ to $N_\infty$
in the GH or $\mathcal{F}$ sense.  In fact
\be
M_j \GHto M_0 \textrm{ and } M_j \Fto M_0
\ee
and
\be
N_j \GHto N_0 \textrm{ and } N_j \Fto N_0
\ee
where $M_0$ and $N_0$ are warped metric spaces defined as in (\ref{MorN})
with warping factor
\be
 f_0(r)=
 \begin{cases}
 1 & r\in[-\pi, 0)
 \\  h_0  & r=0
 \\ 1 &r\in (0, \pi] .
 \end{cases}
\ee
\end{ex}

\subsection{Necessity of $k=n-1$ for $1 \le p < \frac{n}{n-1}$ in the Trace Inequality}\label{subsec:TraceEx}

The next example demonstrates that the conclusion of Theorem \ref{SobolevTraceTrick} does not hold for integer $k$ less than $n-1$ and exponent $p$ in the range $1\le p<\frac{n}{n-1}$. 

\begin{ex}
 Consider the Riemannian manifold $(D,\delta)$, where $D$ is the two dimensional unit disk centered at the origin and $\delta$ is the flat metric. Let $\gamma$ be a vertical line segment in $D$ of length $\frac{1}{2}$ defined by
\begin{equation}
\gamma(t)=(0,t)\text{ for $t$ }\in[-1/4,1/4].
\end{equation}
We define the function $f$ as follows. Let
\begin{equation}
f(w)=\log\left(d(w,\gamma)\right) \text{ for w }\in\lbrace(x,y):0\le x\le\frac{1}{4},\abs{y}\le\frac{1}{4}\rbrace
\end{equation}
and zero otherwise. We calculate that
\begin{equation}
\int_{D}\abs{f}^q<\infty.
\end{equation}
However, we see that $f$ can not have a trace on $\gamma$ with finite $L^{p}$ norm for any $p$. This example can be generalized to higher dimensional situations. The main ingredient is to take some function of the distance from points $x$ to the line segment $\gamma$. 
\end{ex} 

\section{$W^{n-1,p}$ Bounds and the Trace Inequality}\label{sec:Trace}

In this section we show how to satisfy the assumption \eqref{MetricTensorUpperBound} of Theorem \ref{HolderCompactness} by assuming a bound on the norm of the metric in $W^{n-1,p}$. This is ultimately used to satisfy the upper bound in \eqref{d_j1} of Theorem \ref{HLS-thm}. We start by proving a Lemma which will be used in Lemma \ref{BaseCase}. The goal is to prove for compact manifolds a similar result to Lemma 1.1.4 and Lemma 1.1.5 of \cite{MS-SobolevMultipliers} which hold on $\R^n$. 
\begin{lem}\label{SetLem}
Let $(M,g)$ be an n-dimensional closed Riemannian manifold and let $\Omega$ be an open set with smooth boundary. There exist constants $d$, $C$, and $P$, which depend only on $(M,g)$, such that for every point $x$ in $\Omega$, we have one of the following: either there exists a number $0<r_x < d$ such that
	\begin{equation}\label{1Option}
		\mathcal{H}^{n-1}_{g}\left(B_{g}(x,r_x)\cap\partial\Omega\right)\ge Cr^{n-1}_{x}	
	\end{equation} 
	or we have
	\begin{equation}\label{2Option}
		\mu_{g}(B_{g}(x,d)\cap\Omega)\ge P d^{n},
	\end{equation}
	where $\mu_{g}$ is the measure induced by $g$ and $\mathcal{H}^{n-1}_{g}$ is the Hausdorff measure induced by $g$.
\end{lem}
\begin{proof}
 Let $r_{M}$ be the injectivity radius of the $(M,g)$. Consider normal coordinates on a geodesic ball of radius $\frac{r_{M}}{2}$ about the point $x$. Let $\delta$ denote the flat metric on a ball given by normal coordinates. We may see that there is a constant $C_{x}\ge 1$ such that 
\begin{equation}\label{g not so different from Euclidean}
	\frac{1}{C_{x}^2}\le\frac{g(v,v)}{\delta(v,v)}\le C_{x}^2,  \: \forall q \in B\left( x, \frac{r_M}{2}\right), \forall v \in T_qM.
\end{equation}
Compactness allows us to choose just finitely many normal neighborhoods $\{B_g(x_1,r),...,B_g(x_k,r)\}$ and a constant $C \ge 1$ such that
\begin{align}\label{MetricComparison}
\frac{1}{C^2} \le \frac{g(v,v)}{\delta(v,v)}\le C^2,  \: \forall q \in B_g(x_i,r), \forall v \in T_qM,
\end{align}
holds on each normal neighborhood $B_g(x_i,r)$, $1 \le i \le k$ where $\delta$ the flat metric induced by normal coordinates.
Now if we let $\bar{r}$ be the Lebesgue number of the cover $\{B_g(x_1,r),...,B_g(x_k, r)\}$, which only depends on $(M,g)$, we know for each $r' \le \bar{r}$ that $B_g(x,r')$ is contained in some $B_g(x_i, r)$, $1 \le i \le k$ and hence \eqref{MetricComparison} applies on $B_g(x,r')$. 

Now we let,
\begin{align}
	d=\min\left\{\frac{r_{M}}{8},\bar{r}\right\},
\end{align}
and calculate that for $\tilde{d} = C d$ we find
\begin{align}\label{BallComparison}
\mu_{\delta}(B_{\delta}(x,\tilde{d}) \cap \Omega) \le \mu_{\delta}(B_g(x,d)\cap \Omega),
\end{align}
since by \eqref{MetricComparison},
\begin{align}\label{BallComparisonDeltag}
B_{\delta}(x,\tilde{d}) = B_{\delta}(x,C d)= B_{C^2 \delta}(x,d) \subset B_g(x,d). 
\end{align}
Now we define the set,
\begin{align}\label{FsetDefinition}
\mathcal{F} = \{B_g(x,d): x \in \Omega \text{ and } \mu_g(B_g(x,d)) \cap \Omega) < P d^n \},
\end{align}
and so for $B_g(x,d) \in \mathcal{F}$ we can calculate,
\begin{align}\label{MeasureComparedd^n}
\mu_{\delta}(B_g(x,d) \cap \Omega) \le C^{n} \mu_g(B_g(x,d)\cap \Omega) <  C^{n} P d^n.
\end{align}
But, by definition we know that $\tilde{d}^n = C^n d^n$ and hence we can choose $P$ to be so small so that 
\begin{align}\label{ChoiceofK}
P\tilde{d}^n < \frac{1}{2} \mu_{\delta}(B_{\delta}(x,\tilde{d})),
\end{align} 
which can be chosen independently of $x$ by the discussion following \eqref{MetricComparison}.
Thus, we have by combining \eqref{BallComparison}, \eqref{MeasureComparedd^n}, and \eqref{ChoiceofK} that,
\begin{align}\label{HalfMeasure}
2 \mu_{\delta}(B_{\delta}(x,\tilde{d} )\cap \Omega) < \mu_{\delta}(B_{\delta}(x,\tilde{d})).
\end{align}
Now notice that,
\begin{align}
r \mapsto \frac{\mu_{\delta}(B_{\delta}(x,r) \cap \Omega)}{\mu_{\delta}(B_{\delta}(x,r))},
\end{align}
is continuous and equal to $1$ for small enough $r$ and less than $\frac{1}{2}$ for $\tilde{d}$ so we must have that there exists a $0 < r_x < d$ so that,
\begin{align}
\mu_{\delta}(B_{\delta}(x,r_x) \cap \Omega) = \frac{1}{2} \mu_{\delta}(B_{\delta}(x,r_x)).
\end{align}
We may now apply Lemma 1.1.4 of \cite{MS-SobolevMultipliers} to observe that
\begin{align}\label{AlmostDesiredConclusion}
\mathcal{H}_{\delta}^{n-1}(B_{\delta}(x,r_x) \cap \Omega) \ge c r_x^{n-1},
\end{align}
and by using \eqref{MetricComparison} we find,
\begin{align}\label{DesiredConclusion}
\mathcal{H}_g^{n-1}(B_g(x,r_x) \cap \Omega) \ge c r_x^{n-1}.
\end{align}
Hence for every $B(x,d) \in \mathcal{F}$ we can observe \eqref{DesiredConclusion} to achieve \eqref{1Option} and if $B(x,d) \not \in \mathcal{F}$ then \eqref{2Option} is automatically satisfied and we have the desired result.
\end{proof}

We now prove an important Lemma which will be used to prove the main theorem of this section.

\begin{lem}\label{BaseCase}
Let $(M,g)$ be a compact Riemannian $n-$manifold, let $\mu_g$ be the volume measure induced by the metric, and let $\mathcal{H}^{n-1}_g$ denote the $n-1$ Hausdorff measure induced by the metric $g$. For any Borel measure, $\mu$, we have
\begin{equation}
\int|u|d\mu\le C\sup_{\lbrace r>0,x\in M\rbrace} \left \lbrace r^{1-n}\mu\left(B(x,r)\right)\right\rbrace ||u||_{W^{1,1}_{g}(M)},\label{BaseSobolev}
\end{equation}
where $C$ depends on the geometry of $(M,g)$.
\end{lem}
\begin{proof}
We begin by establishing the lemma for smooth functions. Let $u\in C^{\infty}(M)$ and
\begin{equation}
N_{t}(u)=\lbrace|u|>t\rbrace.
\end{equation} 
It is an elementary fact from measure theory that
\begin{equation}
\int|u|d\mu=\int_{0}^{\infty}\mu\left(N_{t}(u)\right)dt.
\end{equation}
If we let 
\begin{equation}
K=\sup_{\Omega}\frac{\mu(\Omega)}{\mu_{g}(\Omega)+\mathcal{H}^{n-1}_{g}(\partial\Omega)},
\end{equation}
then we see that
\begin{equation}
\int_{0}^{\infty}\mu\left(N_{t}(u)\right)\le K\int_{0}^{\infty}\mu_{g}\left(N_{t}(u)\right)+\mathcal{H}^{n-1}_{g}\left(\partial N_{t}(u)\right)dt.
\end{equation}
From the co-area formula, we may deduce that
\begin{equation}
\int_{0}^{\infty}\mu_{g}\left(N_{t}(u)\right)+\mathcal{H}^{n-1}_{g}\left(\partial N_{t}(u)\right)dt=||u||_{W^{1,1}_{g}(\Omega)}\le ||u||_{W^{1,1}_{g}(M)}.
\end{equation}
Thus, if we can show that
\begin{equation}
K\le C\sup_{\lbrace r>0,x\in M\rbrace} \left \lbrace r^{1-n}\mu\left(B(x,r)\right)\right\rbrace ,
\end{equation}
then we have proven the lemma.
Let $d$ be defined by
\begin{equation}
d=\frac{1}{8}(\text{injectivity radius of }(M,g)).
\end{equation}
Now, given an open subset $\Omega$, by Lemma \ref{SetLem} we consider two families of balls,
\begin{align}
\mathcal{F}_{1}&=\lbrace B_g(x,d):\mu_g(B_g(x,d)\cap \Omega) \ge P d^n\rbrace,\\
\mathcal{F}_{2}&=\lbrace B(x,r_x):r_x < d \text{ and } \mathcal{H}_g^{n-1}(B_g(x,r_x) \cap \partial \Omega) \ge c r_x^{n-1}\rbrace.
\end{align}
Using the Besicovitch covering theorem, we may find a countable collections $\lbrace B_g(x_i,d)\rbrace\subset \mathcal{F}_1$ and $\lbrace B_g(x_j,r_{x_j})\rbrace\subset \mathcal{F}_2$ such that elements in $\lbrace B_g\left(x_i,d\right)\rbrace\cup\lbrace B_g\left(x_j,r_{x_j}\right)\rbrace$ are mutually disjoint and
\begin{equation}
\Omega\subset\bigcup_{i=1}^{\infty} B_{g}(x_i,3d)\cup\bigcup_{j=1}^{\infty} B_{g}(x_{j},3r_{x_j})
\end{equation}
We now calculate
\begin{equation}\label{FirstMeasureEst}
\mu(\Omega)\le\sum_{i=1}^{\infty}\mu(B(x_i,3d))+\sum_{j=1}^{\infty}\mu\left(B(x_{j},3r_{x_j})\right).
\end{equation}
Now we also notice that for $B_g(x,d) \in \mathcal{F}_1$,
\begin{align}
\mu(B_g(x,d)) &\le \sup_{\lbrace r > 0, c \in M\rbrace} \left \lbrace r^{1-n}\mu (B_g(c,r))\right\rbrace d^{n-1} \label{PrevMeasureIneq}
\\&\le \sup_{\lbrace r > 0, c \in M\rbrace} \left \lbrace r^{1-n}\mu (B_g(c,r))\right\rbrace \frac{1}{d} d^n 
\\&\le \sup_{\lbrace r > 0, c \in M \rbrace} \left \lbrace r^{1-n}\mu (B_g(c,r))\right\rbrace \frac{1}{P d} \mu_g(B_g(x,d)\cap \Omega)
\\&\le \sup_{\lbrace r > 0, c \in M \rbrace} \left \lbrace r^{1-n}\mu (B_g(c,r))\right\rbrace \frac{1}{P d} \mu_g(B_g(x,d)).\label{MeasureIneq}
\end{align}
It follows by combining Lemma \ref{SetLem}, \eqref{PrevMeasureIneq}, and \eqref{MeasureIneq} that the right side of \eqref{FirstMeasureEst} is bounded by
\begin{align}
&\sup_{\lbrace r>0,x\in M \rbrace}\left\lbrace\frac{\mu\left(B_g(x,r)\right)}{r^{n-1}}\right\rbrace\left(C\sum_{i=1}^{\infty}\frac{1}{3d}\mu_{g}(B_g(x_i,3d))+\sum_{j=1}^{\infty}(3r_{x_j})^{n-1}\right)
\\& \le C_1\sup_{\lbrace r>0,x\in M\rbrace}\left\lbrace\frac{\mu\left(B_g(x,r)\right)}{r^{n-1}}\right\rbrace
\\&\cdot \left(\sum_{i=1}^{\infty}\frac{1}{d}\mu_{g}(B_g\left(x_i,d\right))+c3^{n-1}\sum_{j=1}^{\infty} \mathcal{H}^{n-1}_{g}\left(B_g\left(x_j,r_{x_{j}} \right)\cap \partial \Omega\right)\right)\label{Doubling}
\\&\le C_2\sup_{\lbrace r>0,x\in M\rbrace}\lbrace r^{1-n}\mu(B_g(x,r))\rbrace \left(\mu_{g}(\Omega)+\mathcal{H}^{n-1}_{g}(\partial\Omega)\right),
\end{align}
where we have used the doubling property of $\mu_g$ in \eqref{Doubling}. This completes the proof. 
\end{proof}

\begin{rmrk}
Notice that the supremum in \eqref{BaseSobolev} of Lemma \ref{BaseCase} can be infinite in general.
\end{rmrk}

We now use this result to prove the following Theorem which is a restatement of Theorem \ref{SobolevTraceTrick}.

\begin{thm}\label{MainSobolevThm}
	Let $(K,g)$ be a smooth compact Riemannian $n$-dimensional manifold, possibly with boundary, and let $p\ge 1$. There exists a constant, denoted $C=C(g,K)$, such that
	\begin{equation}
		\left(\int_{\gamma}\abs{f}^pdt\right)^{\frac{1}{p}}\le C\abs{\abs{f}}_{W^{n-1,p}_{g}(K)},
	\end{equation}
	for length minimizing geodesics $\gamma$ with respect to $g$. 
\end{thm}
\begin{proof} In the case where $K$ is compact with boundary we may assume that $(K,g)$ is isometrically embedded in a smooth compact Riemannian manifold without boundary, denoted $(M,h)$, by a collar neighborhood and doubling argument (See p. 208-209 of Bray's RPI paper \cite{Bray-Penrose}). In this case
\begin{align}
M = K \sqcup \partial K \times [-\delta, \delta] \sqcup K
\end{align}
where $\delta >0$. The metric $h$ can be chosen so that $h = g$ on both disjoint pieces of $K \subset M$ and $h(x,-t) = h(x,t)$ on $\partial K \times [-\delta,\delta]$. Now we notice that given a function $f \in W^{n-1,p}_{g}(K)$ we can extend it to $M$ by a standard extension argument in so that
\begin{align}
\abs{\abs{f}}_{W^{n-1,p}_{h}(M)} \le C\abs{\abs{f}}_{W^{n-1,p}_{g}(K)}
\end{align}
and hence it will be sufficient to make most of the remaining argument on $M$. 

 In the case where $K$ was already compact without boundary we will now denote it by $(M,h)$ for consistent notation. Now by Lemma \ref{BaseCase} we see that generally we have for smooth $\phi$
\begin{equation}
	\int_M\abs{\phi}d\mu\le C\sup_{\lbrace r>0,x\in M\rbrace}r^{1-n}\mu(B(x,r))\abs{\abs{\phi}}_{W^{1,1}(M)}\label{MainTraceEq}
\end{equation}
for an arbitrary Borel measure $\mu$ on $M$.

We now generalize the inductive step in the proof of Theorem 1.1.1 in \cite{MS-SobolevMultipliers}. For the manifold $(M,h)$ we know there is a Green's function for the Laplacian (see \cite{H-NonlinearAnalysis}), which we will denote by $G(x,y)$. Distributionally, $G$ satisfies
\begin{equation}
	\Delta_{y}G(x,y)=Vol(M)^{-1}-\delta_{x}(y).
\end{equation}
In addition we know that 
\begin{align}
\abs{G(x,y)}\le Cr^{2-n},\label{GEstimate1}
\end{align}
where $r = d_h(x,y)$, the Riemannian distance from $x$ to $y$ with respect to $h$.
 Thus, after an integration by parts, we have
\begin{align}
	\phi(x)&=Vol(M)^{-1}\int_{M}\phi(y)d\mu_{h}(y)
	\\&+\lim_{\epsilon\rightarrow 0}\left(\int_{M-B(x,\epsilon)}\langle\nabla\phi,\nabla G\rangle d\mu_h(y)+\int_{\partial B(x,\epsilon)}G\frac{\partial\phi}{\partial\nu}dS_{\mu_{h}}(y)\right)
\\&=	Vol(M)^{-1}\int_{M}\phi(y)d\mu_{h}(y)+\int_{M}\langle\nabla\phi,\nabla G\rangle d\mu_{h}(y),\label{GEstimate2}
\end{align}
where $dS_{\mu_h}$ is the measure on $\partial B(x,\epsilon)$ induced by $\mu_h$.

Using estimates for $G$, \eqref{GEstimate1} and \eqref{GEstimate2}, we may now follow the proof of Theorem 1.1.1 and 1.1.2  in \cite{MS-SobolevMultipliers} nearly word for word to see that
\begin{equation}\label{Fundamental Trace Inequality}
	\int\abs{\phi}d\mu\le C\sup_{\lbrace r>0, x\in M\rbrace}\left \lbrace r^{m-n}\mu(B(x,r))\right \rbrace \abs{\abs{\phi}}_{W^{m,1}_{h}(M)},
\end{equation}
for smooth functions $\phi$. In fact we will give the proof in the more general setting where $p \ge 1$ below but first we make an observation about \eqref{Fundamental Trace Inequality}.

For a curve $\gamma$, we let $\mu_\gamma$ be $\mathcal{H}^{1}_{\gamma}$, the one dimensional Hausdorff measure restricted to $\gamma$. Let $\gamma$ be a length minimizing geodesic with respect to $h$ and remember the inclusion $i: \partial K \xhookrightarrow{} M$. By the fact that $\partial K, M$ are compact and $\partial K$ is smoothly embedded into $M$ we have that the Lipschitz constant of the map $i$ is bounded. This implies that
\begin{equation}
	\sup_{\lbrace r>0, x\in M\rbrace}\lbrace r^{-1}\mu_{\gamma}(B(x,r))\rbrace\le C,\label{MeasureBallInequality}
\end{equation}
since $B(x,r)$ is the collection of length minimizing geodesics with respect to $h$ emanating from $x$ with arc length less than $r$ and the distances with respect to $h$ and $g$ are comparable by the bound on the Lipschitz constant of $i$. Thus, we have for all smooth functions $\phi$
\begin{equation}
\int_{\gamma}\abs{\phi}dt\le C\abs{\abs{\phi}}_{W^{n-1,1}_{h}(M)}.
\end{equation}
Thus, by density, the above inequality holds for all functions in $W^{n-1,1}_{h}(M)$. 

We now prove \eqref{Fundamental Trace Inequality} in the case where $p \ge 1$. First from the chain rule, Young's inequality, and \eqref{MainTraceEq}, we have
\begin{align}
	\int_{M}\abs{\phi}^pd\mu=\int_{M}\abs{\phi^p}d\mu&\le C\sup r^{1-n}\mu(B(x,r))\int_M\abs{\nabla\phi^p} + \abs{\phi^p}d\mu_g
	\\&\le C\sup r^{1-n}\mu(B(x,r))\int_{M}\abs{\phi}^p+\abs{\nabla\phi}^p,
\end{align}
which is the base case of the desired inequality.
We now proceed by induction on the number of derivatives of $\phi$, with only slight modification to the proof of Theorem 1.1.1 of \cite{MS-SobolevMultipliers}. Note by \eqref{GEstimate2} we find
\begin{equation}\label{Beginning of Inductive Step}
\int_M\abs{\phi}^p d \mu \le\int_M\abs{Vol(M)^{-1}\int_M\phi d\mu_{h}+\int_M\langle\nabla G,\nabla\phi\rangle d\mu_{h}}^{p}d\mu.
\end{equation}\label{Middle of Inductive Step}
Using Jensen's inequality on the first term and H\"{o}lder's inequality on the second term, we see that the integrand of \eqref{Beginning of Inductive Step} is bounded by
\begin{align}
&\abs{Vol(M)^{-1}\int_M\phi d\mu_{h}+\int_M\langle\nabla G,\nabla\phi\rangle d\mu_{h}}^{p}
\\&\le C(p) \left[Vol(M)^{-1}\int_M\abs{\phi}^pd\mu_h+ \lp\int_M|\nabla G|^{\frac{p-1}{p}} |\nabla G|^{1/p}|\nabla\phi| d\mu_{h}\rp^{p}\right]
\\&\le C(p)Vol(M)^{-1}\int_M\abs{\phi}^pd\mu_h
\\&+ C(p)\lp \lp\int_M|\nabla G| d \mu_h \rp^{\frac{p-1}{p}} \lp \int_M|\nabla G||\nabla\phi|^p d\mu_{h} \rp ^{1/p}\rp^{p}
\\ &\le	C(p) \left[Vol(M)^{-1}\int_M\abs{\phi}^pd\mu_h+\abs{\abs{\nabla G}}^{p-1}_{L^{1}(M)}\int_M|\nabla G||\nabla\phi|^pd\mu_{h}\right].
\end{align}
Using Tonelli's theorem, we may estimate 
\begin{align}
\int_M\abs{\phi}^p d \mu &\le	C(p)\frac{\mu(M)}{Vol(M)}\abs{\abs{\phi}}^{p}_{L^{p}(M)}\label{FirstTermTonneli}
\\&+C(p)\abs{\abs{\nabla G}}^{p-1}_{L^{1}(M)}\int_M\abs{\nabla\phi}^p\left(\int_M\abs{\nabla G}d\mu\right)d\mu_h.\label{SecondTermTonneli}
\end{align}
We can begin by estimating \eqref{FirstTermTonneli} by noticing for $R = 2$diam$(M)$ that,
\begin{align}
\mu(M) = \frac{\mu(B(x,R))}{R^{n-m}} R^{n-m} \le \left[\sup_{\{r > 0, x \in M\}} \frac{\mu(B(x,r))}{r^{n-m}} \right] \max \{1,R^{n-m}\},
\end{align}
and so if we set $m=n-1$, $\mu=\mu_{\gamma}$, and combine with \eqref{MeasureBallInequality} we find
\begin{align}
C(p)\frac{\mu(M)}{Vol(M)}\abs{\abs{\phi}}^{p}_{L^{p}(M)} \le C \abs{\abs{\nabla\phi}}^{p}_{W^{n-1,p}_{h}(M)},
\end{align}
which gives the desired estimate of the right hand side of \eqref{FirstTermTonneli}.

Now we can estimate \eqref{SecondTermTonneli} by applying the inductive hypothesis with the measure $\int_M\abs{\nabla G}d\mu d \mu_h$ acting as $\mu$ to find
\begin{align}
&C(p)\abs{\abs{\nabla G}}^{p-1}_{L^{1}(M)}\int_M\abs{\nabla\phi}^p\left(\int_M\abs{\nabla G}d\mu\right)d\mu_h	\le C\abs{\abs{\nabla G}}^{p-1}_{L^{1}(M)}
\\&\cdot \left(\sup_{\lbrace r>0,x\in M\rbrace}r^{m-n-1}\int_{B_r}\int_{M}\abs{\nabla G}d\mu d\mu_{h}\right)\abs{\abs{\nabla\phi}}^{p}_{W^{n-1,p}_{h}(M)}.	
\end{align}
Now we can estimate $\abs{\nabla G}$ using 
\begin{align}
|\nabla_y G(x,y)| \le Cr^{1-n},\label{GEstimate3}
\end{align}
 in the exact same way as is found in the proof of Theorem 1.1.1 in \cite{MS-SobolevMultipliers}, by splitting the integral,
 \begin{align}
 \int_{B_r}\int_{M}\abs{\nabla G}d\mu d\mu_{h} & = \int_{B_r}\int_{B_{2r}}\abs{\nabla G}d\mu d\mu_{h} + \int_{B_r}\int_{M\setminus B_{2r}}\abs{\nabla G}d\mu d\mu_{h},\label{LastEstimateEq}
 \end{align}
 and noticing that by Tonelli's theorem and \eqref{GEstimate3} we can estimate the first term in \eqref{LastEstimateEq},
 \begin{align}
 \int_{B_r(x)}\int_{B_{2r}(x)}&\abs{\nabla G}d\mu d\mu_{h} \le  C \int_{B_{2r}(x)}\int_{B_{r}(x)}d_h(y,z)^{1-n}d\mu_{h}(z)d\mu(y)  
 \\&\le  C \int_{B_{2r}(x)}\int_0^r\int_{\partial B_{s}(x)}d_h(x,z)^{1-n}dS_{\mu_{h}} ds d\mu(y)
 \\&\le C r \mu (B(x,2r)),
 \end{align} 
 where $dS_{\mu_{h}}$ is the surface measure restricted from $\mu_h$. We can estimate the second term of \eqref{LastEstimateEq},
 \begin{align}
\int_{B(x,r)}\int_{M\setminus B(x,2r)}&\abs{\nabla G}d\mu d\mu_{h} \le   \int_{B(x,r)} \int_{M\setminus B(x,2r)}d_h(y,z)^{1-n}d\mu d \mu_h
\\&\le  C r^n \int_{M\setminus B(x,2r)}d_h(x,z)^{1-n}d\mu
\\&\le  C  r^n\int_{2r}^{\infty} \mu\{2r \le d_h(x,z) \le t\}t^{-n}dt
\\&\le  C  r^n\int_{2r}^{\infty}t^{n-m} t^{m-n} \mu(B(x,t))t^{-n}dt
\\&\le  C  r^n\int_{2r}^{\infty} t^{n-m} \sup_{\{r >0, x \in M\}}\{r^{m-n} \mu(B(x,r))\}t^{-n}dt
\\&\le  C \sup_{\{r >0, x \in M\}}\{r^{m-n} \mu(B(x,r))\} r^n\int_{2r}^{\infty} t^{-m} dt
\\&\le  C  r^{n-m+1} \sup_{\lbrace r >0, x \in M\rbrace}\{r^{m-n} \mu(B(x,r))\}.
 \end{align}
Putting everything together this implies
\begin{align}
\sup_{\lbrace r>0,x\in M\rbrace}r^{m-n-1}\int_{B_r}\int_{M}\abs{\nabla G}d\mu d\mu_{h}\le \sup_{\lbrace r>0,x\in M\rbrace}r^{m-n} \mu(B(x,r))
\end{align}
 and hence when we choose $m=n-1$ and $\mu = \mu_{\gamma}$ we find by \eqref{MeasureBallInequality} that
 \begin{align}
 \int_{\gamma}\abs{\phi}^pd\mu \le C \abs{\abs{\nabla\phi}}^{p}_{W^{n-1,p}_{h}(M)}.
 \end{align}
\end{proof}

\section{Proof of Main Theorem} \label{sec: MainThmProof}

In this section we complete the proof of Theorem \ref{SobolevBoundsAndHolderCompactness} by showing that when the hypotheses of Theorem \ref{SobolevBoundsAndHolderCompactness} are combined with Theorem \ref{SobolevTraceTrick} we can deduce the H\"{o}lder bounds necessary to apply Theorem \ref{HLS-thm}.
\begin{thm}\label{HolderCompactness}
Let $M_j=(M,g_j)$ be a sequence of compact Riemmanian manifolds and $M_0=(M,g_0)$ a background Riemannian manifold. If 
\begin{align}
\int_{\gamma_{q_1q_2}} |g_j|_{g_0}^p dt \le C,\label{MetricTensorUpperBound}
\end{align}
$p > 1$, where $\gamma_{q_1q_2}$ is a length minimizing geodesic connecting $q_1,q_2 \in M$ with respect to $g_0$, $C > 0$ independent of $j,q_1,q_2$. In addition, assume
\begin{align}
g_j(v,v) \ge c g_0(v,v), \: \forall q \in M, v \in T_qM, c > 0, \label{MetricTensorLowerBound}
\end{align}
then there exists a subsequence which converges in the uniform and GH sense to a length space $M_{\infty}=(M,d_{\infty})$ so that
\begin{align}
\sqrt{c} d_0(q_1,q_2) \le d_{\infty}(q_1,q_2) \le C^{1/p} d_0(q_1,q_2)^{\frac{p-1}{p}}.
\end{align}
If in addition, there exists a uniform volume bound,
\begin{align}
Vol(M_j) \le C,
\end{align}
then the sequence will converge in the SWIF sense to the same space as the uniform and GH limit with a possibly different current structure than the sequence (up to possibly taking a metric completion of the SWIF limit to get the GH and uniform limit).
\end{thm}
\begin{proof}
Notice that if $\gamma_{q_1q_2}$ is unit speed with respect to $g_0$ then \eqref{MetricTensorUpperBound} implies
\begin{align}
d_{g_j}(q_1,q_2) &\le \int_{\gamma_{q_1q_2}} \sqrt{g_j(\gamma_{q_1q_2}',\gamma_{pq}')} dt 
\\&\le \int_{\gamma_{q_1q_2}} |g_j|_{g_0} dt
\\&\le \left(\int_{\gamma_{q_1q_2}} dt \right)^{\frac{p-1}{p}} \left(\int_{\gamma_{q_1q_2}} |g_j|_{g_0}^pdt \right)^{1/p} \le C^{1/p} d_0(q_1,q_2)^{\frac{p-1}{p}}.
\end{align}
Similarly, if $\alpha_{q_1q_2}$ is the length minimizing geodesic between $q_1,q_2 \in M$ with respect to $g_j$ then \eqref{MetricTensorLowerBound} implies
\begin{align}
d_{g_j}(q_1,q_2)&= \int_{\alpha_{q_1q_2}} \sqrt{g_j(\alpha_{q_1q_2}',\alpha_{q_1q_2}')} dt 
\\& \ge \sqrt{c}\int_{\alpha_{q_1q_2}} \sqrt{g_0(\alpha_{q_1q_2}',\alpha_{q_1q_2}')} dt \ge \sqrt{c} d_0(q_1,q_2).
\end{align}

Hence we have the uniform H\"{o}lder bounds on the distance functions required to apply Theorem \ref{HLS-thm} to finish up the proof.

\end{proof} 

\textbf{Proof of Theorem \ref{SobolevBoundsAndHolderCompactness}:}
\begin{proof}
By applying Theorem \ref{SobolevTraceTrick} to the function $|g_j|_{g_0}$ we satisfy the metric upper bound assumption of Theorem \ref{HolderCompactness}. Notice by the Sobolev inequality we can find a uniform $L^n$ bound on $|g_j|_{g_0}$ which implies a uniform volume bound on $M_j$. Since we have assumed in Theorem \ref{SobolevBoundsAndHolderCompactness} the other assumption of Theorem \ref{HolderCompactness} directly we find that the desired result follows.
\end{proof}

\section{Appendix by Brian Allen and Christina Sormani} \label{sec: App}

In this section we prove the following theorem which is a restatement of Theorem \ref{HLS-thm} from section \ref{sec: Intro}:

\begin{thm}\label{HLS-thm-2}
Let $M$ be a fixed compact, $C^0$ manifold with piecewise continuous metric tensor $g_0$ and $g_j$ a sequence of piecewise continuous Riemannian metrics on $M$, defining distances $d_j$ such that
\be\label{d_j}
\frac{1}{\lambda}d_0(p,q) \le  d_j(p,q) \le \lambda d_0(p,q)^{\alpha},
\ee
for some $\lambda \ge 1$ where $\alpha \in (0,1]$ then there exists a subsequence of $(M, g_j)$
that converges in the uniform and Gromov-Hausdorff sense. In addition, if there is a uniform upper bound on volume
\be \label{HLS-thm-vol}
\vol(M, g_j) \le V_0.
\ee 
the subsequence converges in the intrinsic flat sense to the same space with a possibly different current structure than the sequence (up to possibly
taking a metric completion of the intrinsic flat limit to get
the Gromov-Hausdorff and uniform limit). 
\end{thm}

The uniform and Gromov-Hausdorff convergence  
under weaker hypothesis was proven by Gromov \cite{Gromov-metric}.  A slightly stronger conclusion was proven in theorem by Huang-Lee-Sormani in the appendix of \cite{HLS}assuming uniform bi-Lipschitz bounds on the $d_j$ with a constructive proof that enabled one to estimate all the uniform, Gromov-Hausdorff and Intrinsic Flat distances precisely.  We will use that construction as a starting point.
To obtain the intrinsic flat convergence with our weaker hypotheses,
we need to make a
careful analysis of the properties of the manifolds and their
limit viewed as integral current spaces.

   We begin with a subsection which 
reviews the notion of an integral current space as defined by Sormani-Wenger in \cite{SW-JDG} and Ambrosio-Kirchheim's
integral currents as in \cite{AK}.  At the end of this review we state
a theorem concerning integral current spaces, Theorem~\ref{HLS-2}, which we can then immediately see implies Theorem~\ref{HLS-thm}.
We conclude with the proof of that theorem using serious analysis from \cite{HLS} and \cite{AK}.

\subsection{Converting our Manifolds to Integral Current Spaces}\label{subsec: Currents Background}

Given a fixed possibly singular compact, connected (we will always assume connected) manifold, $(M,g_0)$ 
and a sequence of possibly singular Riemannian metrics $g_j$
as in Theorem~\ref{HLS-thm}  then one can define a sequence of distances
\be
d_j(p,q)=\inf \{L_j(C):\,\, C(0)=p, \, C(1)=q\}
\ee
where
\be
L_j(C)=\int_0^1 g_j(C'(s),C'(s))^{1/2}\, ds.
\ee
Observe that by the definition of $d_j$ and the definition of
the smooth structure on $M$, that any diffeomorphic
chart 
\be
\varphi: U\subset {\mathbb R}^m \to \varphi(U) \subset M
\ee
such that $\varphi(U)$ avoids a singular point in $(M,g_j)$ can be
restricted to a compact set $K\subset U$ so that
\be
\varphi: K \subset {\mathbb R}^m \to \varphi(K) \subset M
\ee
is bi-Lipschitz with respect to $d_j$ and $d_0$, although there
is no uniform bound on this bi-Lipschitz constant.   We know that
balls measured with respect to $d_j$ are open with respect to the 
manifold topology.  A set is of $m$-dimensional Hausdorff measure zero with respect
to $d_j$ iff it has $m$-dimensional Hausdorff measure zero with respect to $d_0$ (Lebesgue measure on $M$):
\be\label{Hausdorff-zero}
\mathcal{H}^m_{d_j}(A)=0 \iff \mathcal{H}^m_{d_0}(A)=0.
\ee
%For the above: HOW SINGULAR CAN WE ALLOW $M$
%AND $g_j$ TO BE????

If a countable collection of disjoint charts on 
$M^m$ that cover $\mathcal{H}_j$-almost the whole space 
are bi-Lipschitz with respect to $d_j$, then we can define an integral current, $T$, on $M$ in the sense of Ambrosio-Kirchheim \cite{AK}
using those charts, so that
\be\label{defT}
T(f, \pi_1,...,\pi_m)=\int_M f\, d\pi_1\wedge\cdots\wedge d\pi_m.
\ee
Note that $T$ does not depend on $j$ except insofar as verifying
that there is a collection of smooth charts on $M$ are bi-Lipschitz with respect to $d_j$
with some bi-Lipschitz constants that need not even be uniform for
a fixed $j$.  Observe also that support
of $T$ satisfies:
\be
\spt(T)=\left\{p\in M: \, \forall r>0 \,\,T\rstr B(p,r) \neq 0\right\}=M.
\ee

Now suppose we have a compact metric space, $X$,
with a sequence of distances, $d_j$, and an integral current
$T$ defined on $X$, such that 
\be
\spt(T)=X.
\ee
  Ambrosio-Kirchheim define the mass measure of $T$ 
in Definition 2.6 of \cite{AK} to be
the smallest Borel measure $\mu$ such that
\be\label{AKmass}
T(f, \pi_1,...,\pi_m) \le \prod_{i=1}^m \Lip_{d_j}(\pi_i) \int_X f d\mu.
\ee
Since the Lipschitz constants $\Lip_{d_j}$ depend on $d_j$, the
mass measure also depends on $d_j$.  So we
will write $||T||_{d_j}$ to indicate the mass measure.  
The set of positive density of $T$ then also depends on $d_j$:
\be
\set_{d_j}(T)=\left\{p\in M:\, \liminf_{r\to 0} \frac{||T||_{d_j}(B(p,r))}{r^m} >0\right\}.\label{setDef}
\ee

In particular, when $X=M^m$ is a smooth manifold with singular points, and $T$ is defined using its charts as in (\ref{defT}),
the mass measure is the volume with respect to $g_j$
\be
\vol_{g_j}(A)=||T||_{d_j}(A),
\ee
where $A \subset M$ a Borel set (See \cite{Sormani-scalar}).
A singular point with a cusp singularity (with respect to $d_j$)
will not be included in $\set_{d_j}(T)$, but a cone singularity
(with respect to $d_j$) will be included in the set (See Example A.8 and A.9 of Sormani-Wenger \cite{SW-JDG} for a discussion of cone and cusp points).  The determination as to whether a singularity is cone or a cusp depends on the
$\vol_{g_j}$ near the singular point.

We can thus define a sequence of integral current spaces,
$(X_j, d_j, T)$ as in Sormani-Wenger \cite{SW-JDG} by
taking 
\be \label{Xjset}
X_j =\set_{d_j}(T) \subset X
\ee
and $d_j$ restricted to $X_j$ from $X$.  Note that in \cite{SW-JDG}
Sormani-Wenger already defined the notion of an integral current space associated to a singular Riemannian manifold using
(\ref{defT}).  What we have done here is observe that when $M$ is
constant and only $g_j$ changes, we obtain an integral current
space in which $(X_j, d_j)$ depend upon $j$ but not $T$.
Note that in Huang-Lee-Sormani \cite{HLS} the space $X$
did not depend on $j$ either  because there 
were uniform bi-Lipschitz controls on the distances in that paper.

Note that Ambrosio-Kirchheim proved that the closure of the set of a current is the support of the current.  
\be
Cl_{d_j}(X_j)=Cl_{d_j}(\set_{d_j}(T))=\spt(T)=X.
\ee
So the metric completion of $(X_j,d_j)$ is $(X,d_j)$.

\subsection{Convergence of Integral Current Spaces}

Intrinsic flat (SWIF) convergence was defined in Sormani-Wenger \cite{SW-JDG} for sequences of integral current spaces, and the limits are also integral current spaces (and thus not necessarily compact).  Sormani-Wenger proved that if a sequence converges in the GH sense, and there is a uniform upper bound on the total mass,
\be
\mass(X_j,d_j,T_j)=||T_j||_{d_j}(X_j),
\ee
and $\partial T_j=0$  then a subsequence converges in the SWIF sense to an integral current space which is either the ${\bf 0}$ space or 
$(X_\infty, d_\infty, T_\infty)$ where $X_\infty$ is subset of the GH limit, $X$, and $d_\infty$ is the restricted metric of the GH limit.

In the Appendix of \cite{HLS}, Huang-Lee-Sormani proved that if one has a fixed compact metric space $(X, d_0)$ and a sequence of metrics, $d_j$ on $X$, such that
\be
\frac{1}{\lambda}d_0(p,q) \le  d_j(p,q) \le \lambda d_0(p,q),
\ee
then a subsequence converges in the uniform, GH, and SWIF sense to the same compact limit space, $(X, d_\infty)$. In particular, no cusp singularities form. Throughout the same integral current is used to define the SWIF distance for all metric spaces. 

Here we assume only uniform H\"older bounds on the distance functions and so cusps may form.  As in \cite{HLS} and \cite{Gromov-metric}, we obtain subsequences which converge in the uniform and GH sense to some $(X, d_\infty)$.  However now our SWIF limit, $X_\infty$, may be a proper subset of the GH limit, $X$.  Nevertheless we do prove $Cl_{d_\infty}(X_\infty)=X$ which is the best one can expect given that cusp singularities may form with only H\"older bounds on the distance functions.  

\begin{thm}\label{HLS-2}
Fix a compact metric space $(X, d_0)$ and fix
$\lambda>0$.   Suppose that
$d_j$ are metrics on $X$ and 
\be\label{d_j'}
\frac{1}{\lambda}d_0(p,q) \le  d_j(p,q) \le \lambda d_0(p,q)^{\alpha}
\ee
where $\alpha\in (0,1]$.
Then there exists a subsequence, also denoted $d_j$,
and a length metric $d_\infty$ satisfying \eqref{d_j'} with
$j=\infty$ such that
$d_j$ converges uniformly to $d_\infty$:
\be\label{epsj}
\epsilon_j= \sup\left\{|d_j(p,q)-d_\infty(p,q)|:\,\, p,q\in X\right\} \to 0.
\ee 
Furthermore
\be\label{GHjlim}
\lim_{j\to \infty} d_{GH}\left((X, d_j), (X, d_\infty)\right) =0
\ee
If in addition  $(X_j, d_j, T)$, $X_j \subset X$, are $m$ dimensional
integral current spaces without boundary such that 
\be
Cl_{d_j}(X_j)=X
\ee
with a uniform upper bound on total mass:
\be\label{total-mass}
||T||_{d_j}(X) \le V_0
\ee
then
\be\label{Fjlim}
\lim_{j\to \infty} d_{\mathcal{F}}\left((X_j, d_j,T), (X_\infty, d_\infty,T_\infty)\right) =0.
\ee
where $(X_\infty, d_\infty, T_\infty)$ has $\set_{d_\infty}(T_{\infty})=X_\infty\subset X$ and
the metric completion of $X_\infty$ with respect to $d_\infty$
is $X$: 
\be
Cl_{d_\infty}(X_\infty)=X.
\ee  
\end{thm}

\begin{rmrk}
It is clear from the previous subsection that this theorem implies Theorem~\ref{HLS-thm} since (\ref{HLS-thm-vol}) implies
(\ref{total-mass}).
\end{rmrk}

\begin{rmrk}
Note that our hypothesis on the distances in (\ref{d_j}) is not strong enough
to control the total mass uniformly (see Example~\ref{notpyramids} in the next subsection). 
By Ambrosio-Kirchheim \cite{AK} Theorem 9.5 and Lemma 9.2 for a rectifiable current,
\be
||T||_{d_j}(X) \le \theta_j  (2^m/\omega_m) \mathcal{H}^m_{d_j}(X)
\ee
where $\theta_j$ is the maximum multiplicity of $(X, d_j, T)$. 
In \cite{HLS} the conditions on $d_j$ were strong enough
to obtain uniform bounds on the Hausdorff measure and consequently on the total mass $||T||_{d_j}(X)$ but our 
hypothesis in (\ref{d_j'}) do not provide uniform bounds
on the Hausdorff measure (see Example~\ref{notpyramids}). 
\end{rmrk}

\begin{rmrk}
In \cite{HLS}, there were bi-Lipschitz relationships between
$d_j$ and $d_\infty$ that allowed one to see that 
$T_\infty= T$.   It would take some serious geometric measure 
theory to find an example satisfying the weaker H\"older conditions
of the $d_j$ demonstrating that we need not have $T_\infty= T$.
The challenge is that one would need to find an example
where the distances behaved badly on a set of positive measure.  
These are a challenge to construct.
\end{rmrk}

The proof of this theorem follows in the style of the appendix of \cite{HLS} with significant extra complications arising due to the weaker hypothesis.

\begin{proof}
Let us examine the distances, 
\be
d_j: X\times X \to [0, \lambda \diam_{d_0}(X)^{\alpha}]
\ee
as functions on the compact metric space
$X\times X$ endowed with the taxi product metric:
\be
d_{taxi}((p_1,p_2),(q_1,q_2)) = d_0(p_1,q_1)+d_0(p_2,q_2).
\ee
By the triangle inequality and (\ref{d_j'}) we have
equicontinuity of the $d_j$ with respect to $d_{taxi}$:
\begin{eqnarray}
|d_j(p_1,p_2)-d_j(q_1,q_2)|
&\le& d_j(p_1,q_1) + d_j(p_2, q_2)\\
&\le& \lambda d_0(p_1,q_1)^{\alpha}
+\lambda d_0(p_2,q_2)^{\alpha}\\
&\le& 2\lambda d_{taxi}((p_1,p_2),(q_1,q_2))^{\alpha}.
\end{eqnarray}
Thus by the Arzela-Ascoli theorem applied to the compact
metric space $(X\times X, d_{taxi})$, there is a subsequence
of the $d_j$, also denoted $d_j$, which converges uniformly to some 
\be
d_\infty: X\times X \to [0, \lambda \diam_{d_0}(X)]
\ee
which satisfies (\ref{d_j'}) with $j=\infty$.  In particular,
$d_\infty$ is a definite, symmetric function satisfying the
triangle inequality.  Thus we have
(\ref{epsj}) and by Gromov \cite{Gromov-metric} we obtain 
Gromov-Hausdorff convergence as in (\ref{GHjlim}) as well,
to a compact metric space $(X, d_\infty)$. 

 By Sormani-Wenger \cite{SW-JDG}, a further subsequence, 
also denoted with just $j$, converges in the SWIF sense to some 
integral current space $(X_\infty, d_\infty, T_\infty)$ where $X_\infty\subset X$ as long as there is a uniform upper bound on the total mass $||T||_{d_j}(X)$. At this point we do not yet
know how $T_\infty$ is related to $T$.  We do not even know if
$T$ is even an integral current structure for $(X_\infty, d_\infty)$
since the charts defining $T$ which were bi-Lipschitz with respect to
$d_0$ need not be bi-Lipschitz with respect to $d_\infty$.   We only
know (\ref{d_j'}) holds with $j=\infty$.  In order to better understand 
$(X_\infty, d_\infty, T_\infty)$ we will not just cite the work in
\cite{SW-JDG} but construct this limit space explicitly using 
the construction in \cite{HLS} combined with the proof 
of a theorem in \cite{SW-JDG}.

In \cite{HLS}, Huang-Lee-Sormani construct an ambient metric space 
\be
Z_j= [-\epsilon_j, \epsilon_j] \times X.
\ee
where
\be\label{epsj}
\epsilon_j= \sup\left\{|d_j(p,q)-d_\infty(p,q)|:\,\, p,q\in X\right\} 
\ee
with a metric $d'_j$ on $Z_j$ such that
\be \label{iso-}
d'_j((-\epsilon_j,p), (-\epsilon_j,q)) = d_j(p,q)
\ee
\be \label{iso+}
d'_j((\epsilon_j,p), (\epsilon_j,q)) = d_\infty(p,q).
\ee
Thus we have metric isometric embeddings
$\varphi_j:(X, d_j)\to (Z_j, d'_j)$ and
$\varphi_j':(X, d_\infty)\to (Z_j, d'_j)$ such that
\be
\varphi_j(p)=(-\epsilon_j, p)
\textrm{ and } \varphi'_j(p)=(\epsilon_j, p).
\ee

We can create an even larger complete metric space, $(Z, d')$ gluing together all these $(Z_j, d'_j)$ along the isometric images
 $\varphi'_j(X)$ following the construction in Section 4.1 of Sormani-Wenger \cite{SW-JDG}.
  There it was shown that
  distance preserving maps
 remain distance preserving maps,
 \be
 \varphi_j: (X, d_j) \to (Z,d') \textrm{ and }
 \varphi'=\varphi_j': (X, d_\infty) \to (Z, d')
 \ee
 where $\varphi_j'$ no longer depend on $j$ when
 viewed as maps into this space $Z$ where their
 images have been identified.  Note that in the specific 
 setting of the \cite{HLS} construction, for
 each $x\in X$ we have identified 
 all the points
 $(x, \epsilon_j)\subset Z_j$ for $j=1,2,...$
 to be a single point in $Z$.   See Figure~\ref{fig-book}.
 
 \begin{figure}[h]
\begin{center}
\includegraphics[width=\textwidth]{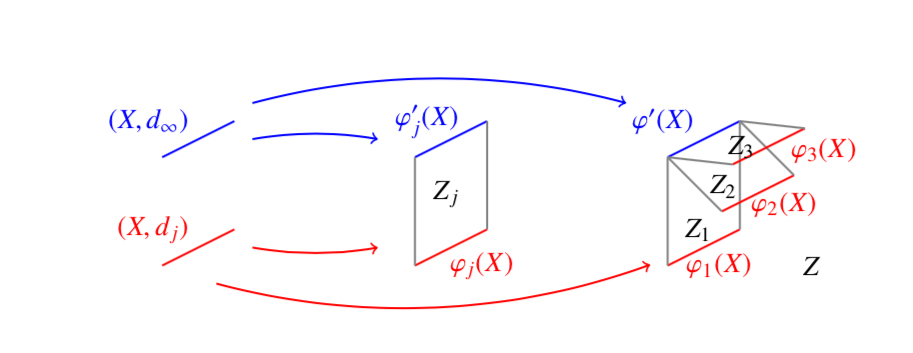} 
\caption{Here $Z$ is a book where 
\textcolor{blue}{$\varphi'(X)=\varphi_j'(X)$} is the spine 
and each page is
 a $Z_j$ running from  \textcolor{red}{$\varphi_j(X)$}
 to \textcolor{blue}{$\varphi_j'(X)$}.}   
\label{fig-book}
\end{center}
\end{figure}
 
  Lying within $(Z, d')$ is a compact metric space
$(Z',d')$ consisting of the union of the images of $\varphi_j'(X)$
and $\varphi_j(X)$.  Applying Ambrosio-Kirchheim compactness theorem \cite{AK}, a subsequence of the $S_j=\varphi'_{j\#}T$ converges weakly to some integral current $S$ in $Z' \subset Z$.  As in Section 4.1 of \cite{SW-JDG}
we can conclude
\be
(X, d_j, T) \Fto (X_\infty, d_\infty, T_\infty)
\ee
where $X_\infty\subset X$ has $\varphi_j'(X_\infty)=\set(S)$
and $\varphi'_{j\#}T_\infty=S$ and $\varphi_j': (X, d_\infty)\to (Z, d')$
is the distance preserving map constructed above. 

In the next few paragraphs we prove
$Cl_{d_\infty}(X_\infty)=X$.  This follows iff
\be
Cl_{d'}(\set(S))= \varphi_j'(X) \subset Z.
\ee
According to Ambrosio-Kirchheim \cite{AK} the closure of the set of positive
density of an integral current is the support:
\be
Cl_{d'}(\set(S))= \spt(S)=\{z \in Z: \,
\forall r>0\,\, ||S||(B(p,r))>0 \}.
\ee
Suppose on the contrary there exists a point $x_0\in X$
such that $z_0=\varphi'_j(x_0) \notin \spt(S)$.  Then there
exists $r>0$ such that
\be
||S||(B(z_0,r))=0.
\ee
Since $S_j=\varphi'_{j\#}T$ converge weakly to $S$,
\be\label{contradict-S}
\lim_{j\to \infty}||S_j||(B(z_0,r))=0.
\ee
We may also view $S_j$ as integral currents in $Z_j \subset Z$
and $B(z_0,r)$ as a ball in $Z_j$ centered on a point
$z_0=\varphi_j'(x_0)\subset Z_j$.   

Examining the definition
of the distance on $Z_j$ (really we only use the Hausdorff distance of
$\varphi_j(X)$ to the $\varphi_j'(X)$ is $<2\epsilon_j$ to get this), 
we know there is a point $x_j\in X$
and $z_j=\varphi_j(x_j) \in Z_j$ such that 
\be
d_j'(z_0, z_j) < 2\epsilon_j.
\ee
So for $j$ sufficiently large that $2\epsilon_j < r/2$ we have
\be
B(z_j, r/2)\subset B(z_0,r).
\ee
Since $\varphi_j: (X, d_j) \to Z_j$ is distance preserving
and $S_j =\varphi_{j\#}T$, we have
\be\label{S_j}
||S_j||(B(z_0,r))\ge ||S_j||( B(z_j, r/2) ) 
= ||T||_{d_j}(B_{d_j}(x_j, r/2)).
\ee
By (\ref{d_j'}), we know
\begin{eqnarray}
B_{d_j}(x_j, r/2)&=&\{ x\in X:\, d_j(x, x_j)<r/2\}\\
&   \subset & \{ x\in X:\ \tfrac{1}{\lambda}d_0(x,x_j) < r/2\}\\
&=&  \{ x\in X:\ d_0(x,x_j) < \lambda r/2\}\\
&=& B_{d_0}(x_j,  \lambda r/2).
\end{eqnarray}
Since $x_j$ lie in a compact metric space $(X, d_0)$, a subsequence
of the $x_j$ converges with respect to $d_0$ to $x_\infty \in X$.
So for $j$ sufficiently large
\be
B_{d_j}(x_j, r/2) \supset B_{d_0}(x_j,  \lambda r/2)
\supset B_{d_0}(x_\infty,  (\lambda r/2)/2).
\ee
Combining this with (\ref{S_j}) and $\|T\|_{d_j} \ge \frac{1}{\lambda^m} \|T\|_{d_0}$ we have
\be
||S_j||(B(z_0,r))\ge ||T||_{d_j}(B_{d_j}(x_j, r/2))
\ge \frac{1}{\lambda^m}||T||_{d_0} (B_{d_0}(x_\infty,  (\lambda r/2)/2))
\ee
which is a constant greater than $0$ because $T$ is 
the integral current structure of $X$. 
This uniform positive lower bound on $||S_j||(B(z_0,r))$
contradicts (\ref{contradict-S}).   Thus $Cl_{d_\infty}(X_\infty)=X$.
\end{proof}

\subsection{An Example with Volume Diverging}

In this section we prove the following example exists: 

\begin{example}\label{notpyramids}
There exists a sequence of
piecewise flat Riemannian metrics, $g_j$, on $[0,1]^2$
whose distances, $d_{g_j}$ satisfy the H\"older condition (\ref{d_j})
which have volumes diverging to infinity but 
converge in the Gromov-Hausdorff sense to 
$([0,1]^2, d_{taxi})$ where 
\be
d_{taxi}((x_1,y_1), (x_2,y_2)) = |x_1-x_2|+|y_1-y_2|.
\ee

\end{example}

Before we construct the example we begin with the basic building
blocks as in Figure~\ref{g-0-g-h}

 \begin{figure}[h]
\begin{center}
\includegraphics[width=.9\textwidth]{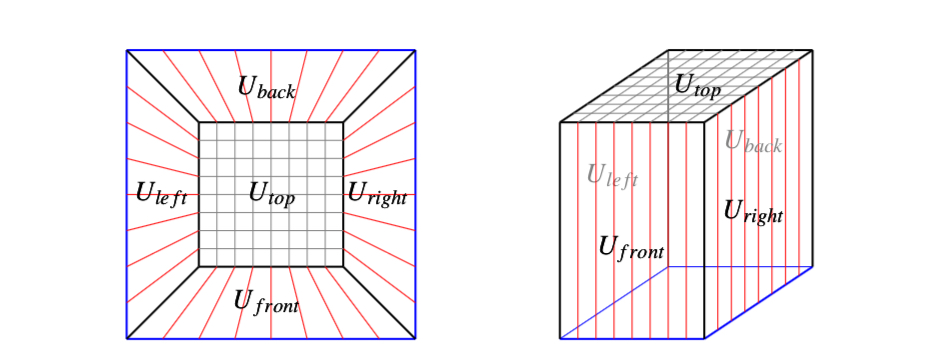} 
\caption{The basic building block for Example~\ref{notpyramids}.  }
\label{g-0-g-h}
\end{center}
\end{figure}

Divide $[0,1]^2$ into five regions:
\be
[0,1]^2=U_{top}\cup U_{front}\cup U_{back}
\cup U_{left} \cup U_{right}
\ee
where  
\begin{eqnarray}
\qquad 
U_{top}\,\,\,\,\,&=&[1/4,3/4]^2
\\
 U_{left}\,\,&=&\{(x,y):\, x\le y\le 1-x\textrm{ and } x \le 1/4\}
\\
 U_{right}&=&\{(x,y):\, 1-x\le y\le x\textrm{ and } x \ge 1/4\}
 \\
U_{front}&=&\{(x,y):\, y\le x\le 1-y\textrm{ and } y \le 1/4\}
\\
 U_{back}\,&=&\{(x,y):\, 1-y\le x\le y\textrm{ and } y \ge 1/4\}
 \end{eqnarray}

We claim that for any $h>1$, there exists  a piecewise flat metric $g_h$ on $[0,1]^2$
which forms the sides and top of a rectangular block of height $h$.
This metric can be defined by pulling back the Euclidean metrics
using the inverses of the following diffeomorphisms
\begin{eqnarray*}
&F_{top}:[0,1]\times [0,1] \to U_{top}
\textrm{ where }
&F_{top}(s,t) =(\tfrac{1}{4} (1-s)+\tfrac{3}{4} s, \tfrac{1}{4} (1-t)+\tfrac{3}{4} t)\\
&F_{left}:  [0,h]\times [0,1]\to U_{left}
\textrm{ where }
&F_{left}(s,t)=(\tfrac{s}{4h}, \tfrac{s}{4h} (1-t) + (1-\tfrac{s}{4h})t)\\
&F_{right}:  [0,h]\times [0,1]\to U_{right}
\textrm{ where }
&F_{right}(s,t)=
(1-\tfrac{s}{4h}, (1-\tfrac{s}{4h}) (1-t) + (\tfrac{s}{4h})t)\\
&F_{front}:  [0,1]\times[0,h] \to U_{front}
\textrm{ where }&
F_{front}(s,t)=( \tfrac{t}{4h}(1-t)+(1-\tfrac{t}{4h}s, \tfrac{t}{4h})\\
&F_{back}:  [0,1]\times[0,h] \to U_{front}
\textrm{ where }&
F_{back}(s,t)=( (1-\tfrac{t}{4h})(1-t)+\tfrac{t}{4h}s, 1-\tfrac{t}{4h})\\
\end{eqnarray*}
Thus
\be\label{volgrid}
Vol_{g_h}([0,1]^2)=1\cdot 1+h\cdot 1 + h\cdot 1+ h\cdot 1+
h\cdot 1=1+4h.
\ee

We claim that the shortest distance between any pair of points
$p$ and $q$ in the boundary of $[0,1]^2$ is achieved by the length
of a curve in the boundary of $[0,1]^2$ measured using the standard Euclidean metric
\be \label{ongrid}
d_{g_h}(p,q) =\inf \{ L_{g_0}(C):\, C[0,1]\to \partial [0,1]^2,
C(0)=p,\, C(1)=q\}.
\ee
First observe that $L_{g_0}=L_{g_h}$ for any curve that
lies in the boundary.  So we need only verify that there aren't
any shorter curves.  Suppose there is a shorter curve $C_{p,q}$.
If that curve enters $U_{top}$ then it must have travelled a distance
$h>1$ to reach the top and again $h>1$ to come back down, thus
it has length $2$.  That is the maximum length of any curve in the boundary, so $C_{p,q}$ cannot reach $U_{top}$.  However if
$C_{p,q}$ lies in $[0,1]^2\setminus U_{top}$, then we can project it
radially to the boundary $\partial [0,1]^2$ to a curve of shorter length
because the map 
\be
\pi: [0,h]\times [0,1]\to [0,1] \textrm{ where } \pi(s,t)=(0,t)
\ee
shortens lengths of curves with respect to the Euclidean metric. So we have our claim.
 
We claim that
\be \label{Haus}
\forall p\in [0,1]^2 \,\,\exists q_p \in \partial [0,1]^2 
\textrm{ such that } d_{g_h}(p,q_p) \le  h + \sqrt{2}.
 \ee
 If $p\subset [0,1]^2\setminus U_{top}$ we just take
 $q_p$ to be the image of the projection of $p$ by $\pi$
 and the distance is $\le h$.  If $p\in U_{top}$ then
 it might need to traverse at most $\diam_{g_h}(U_{top})=\sqrt{2}$
 further.
 
 We now define $X=[0,1]^2$ and define metric tensors
 $g_j$ on $X$ by first dividing $X$ into $2^j\cdot 2^j$ squares:
 \be
 X =\bigcup_{l,m=1}^{2^j} S^j_{l,m}
 \textrm{ where }
 S^j_{l,m}=
  [\tfrac{l-1}{2^j},\tfrac{l}{2^j}]\times
  [\tfrac{m-1}{2^j},\tfrac{m}{2^j}].
  \ee
 Note that we have diffeomorphisms: 
 \be
 F_{l,m}^j: S^j_{l,m}\to [0,1]^2 \textrm{ where }
 F_{l,m}^j(x,y)=(2^j x - 1, 2^j y -1).
 \ee
 We define the metric tensor $g_j$ by pulling back $g_h$
 for $h=h_j$ and rescaling it back down on each square:
  \be
 g_j= (1/2^j)^2 F^{j*}_{l,m}g_{h_j} \textrm{ on } S^j_{l,m} 
  \ee
  as depicted in Figure~\ref{g-1-g-2}.

 \begin{figure}[h]
\begin{center}
$(M,g_1)$\includegraphics[width=.2\textwidth]{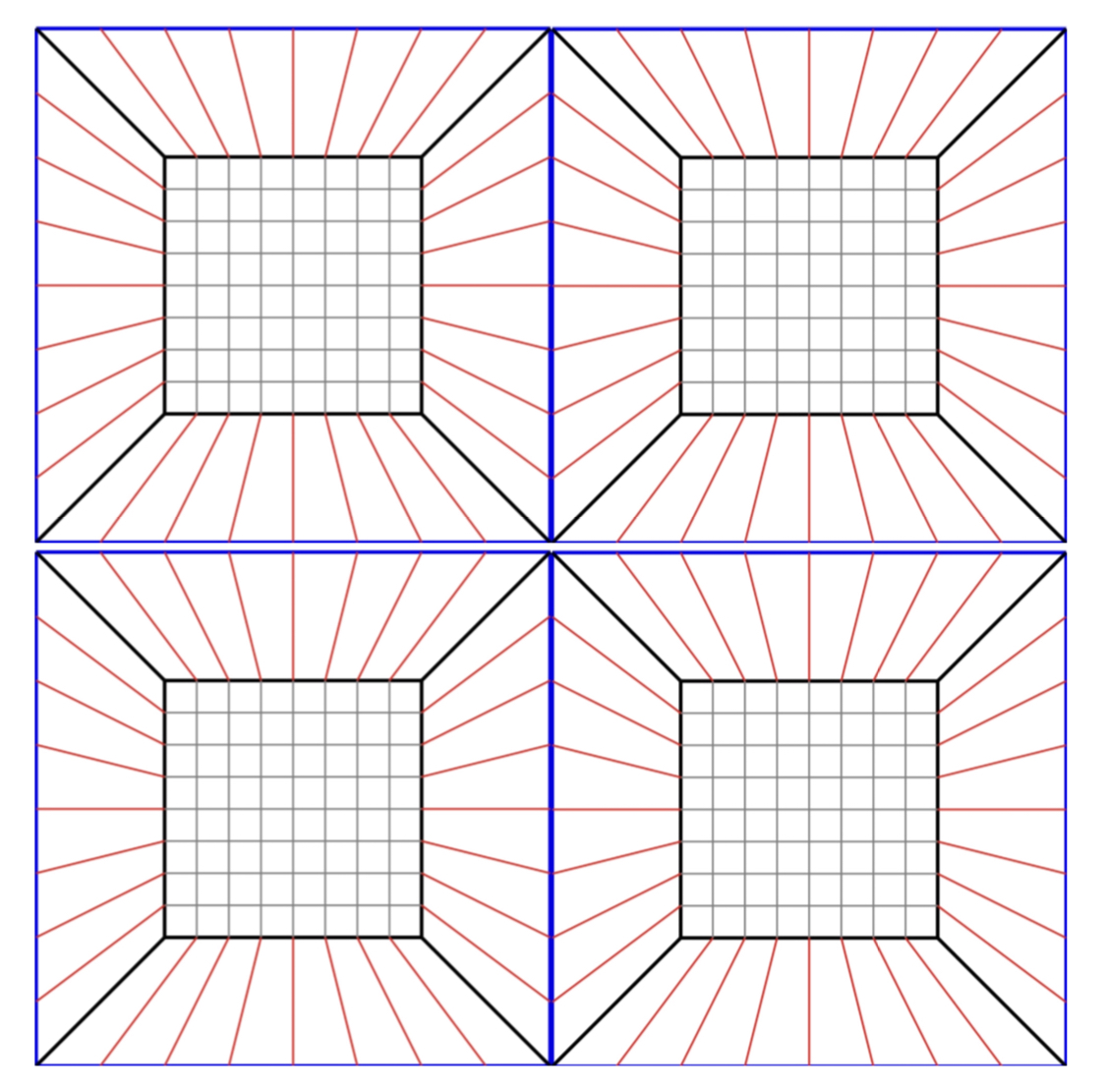} 
$(M,g_2)$\includegraphics[width=.2\textwidth]{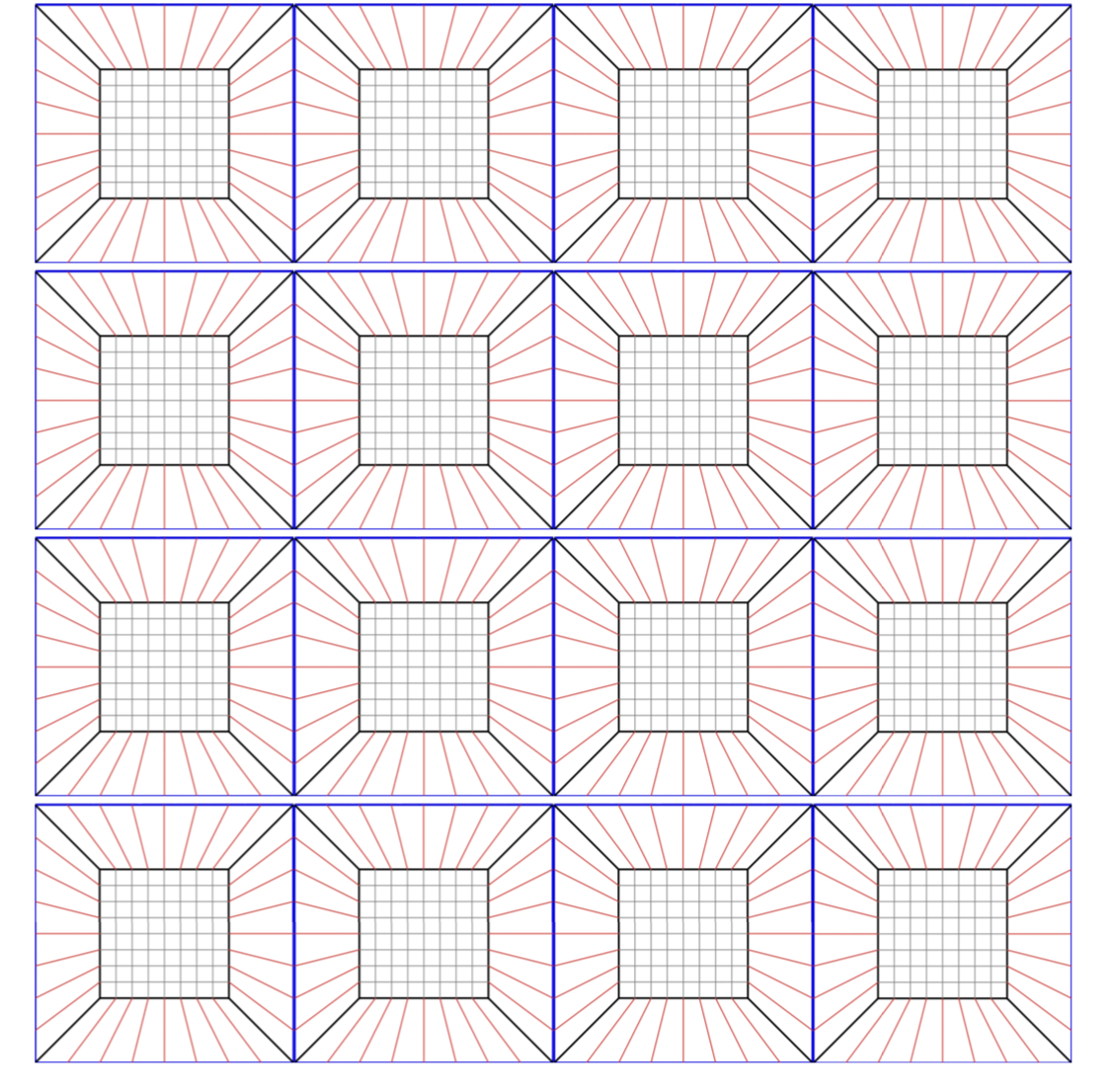} 
$(M,g_3)$\includegraphics[width=.2\textwidth]{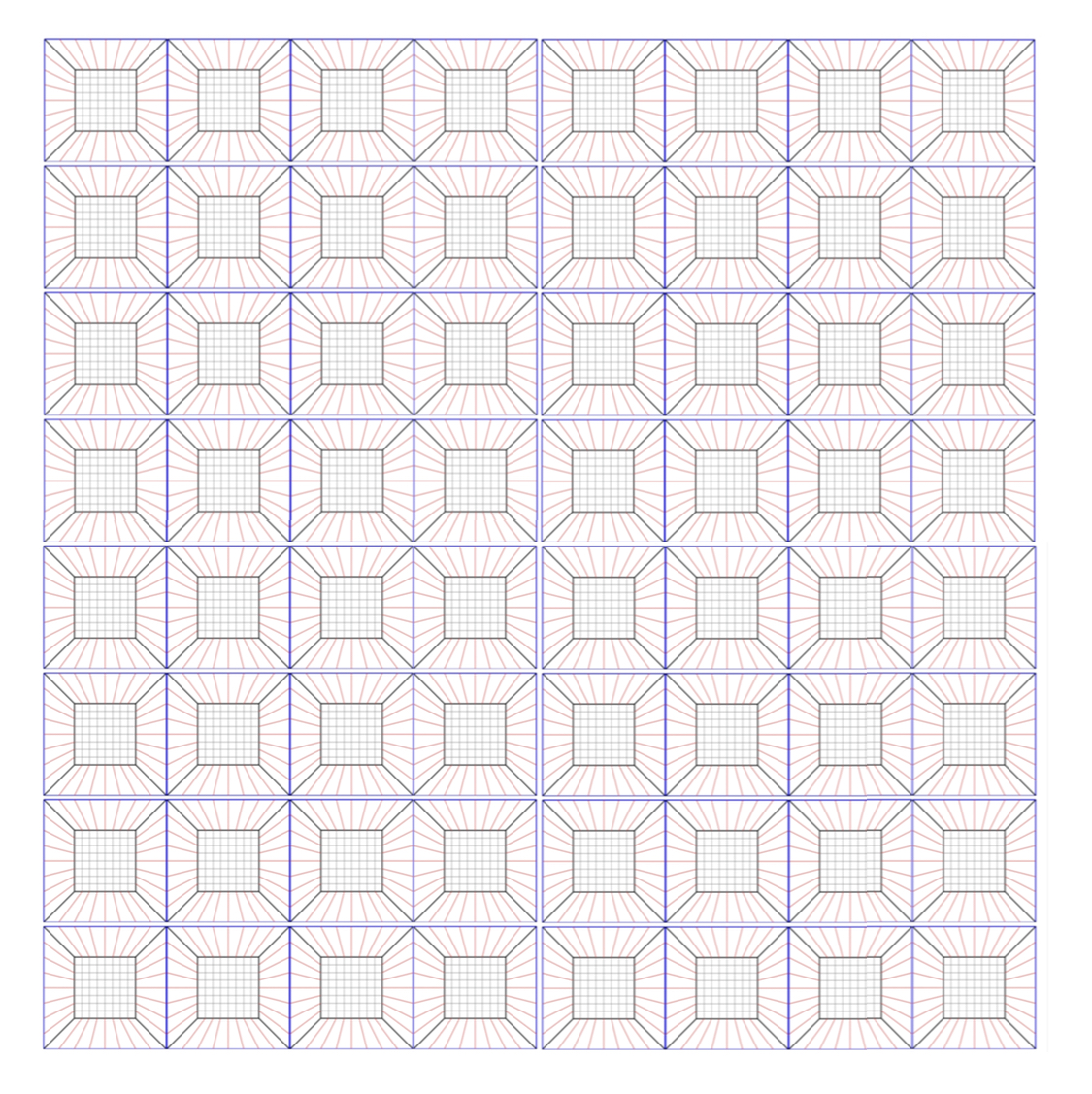} $\qquad$
\caption{The first few terms in the sequence $([0,1]^2, g_j)$.  }
\label{g-1-g-2}
\end{center}
\end{figure}

By (\ref{volgrid}) we have 
 \be\label{volgrid1}
Vol_{g_j}(S^j_{l,m})=(1/2^j)^2 (1+4h_j)
\ee
and so
\be\label{volgrid2}
Vol_{g_j}(X)=2^{2j}(1/2^j)^2 (1+4h_j)= 1+ 4h_j.
\ee

By (\ref{ongrid}) we have for all $p,q \in \partial S^j_{l,m}$,
 \be \label{ongrid1}
d_{g_j}(p,q) =\inf \{ L_{g_0}(C):\, C[0,1]\to \partial S^j_{l,m},
C(0)=p,\, C(1)=q\}.
\ee
So if we define the grid
\be
X_j= \bigcup_{l,m=1}^{2^j} \partial S^j_{l,m} \subset X
\ee
then
\be\label{ongrid2}
\forall p,q \in X_j, \,\, d_{g_j}(p,q)=d_{taxi}(p,q).
\ee
 By (\ref{Haus}) we have 
 \be \label{Haus1}
\forall p\in X \,\,\exists q_p \in X_j 
\textrm{ such that } d_{g_j}(p,q) \le  (1/2^j)(h_j + \sqrt{2}).
 \ee
  Combining this we have 
 \be
 d_{GH}((X,d_{g_j}), (X_j, d_{taxi}) )  \le (1/2^j)(h_j + \sqrt{2}).
 \ee
 It is also easy to see that
 \be \label{Haus2}
\forall p\in X \,\,\exists q_p \in X_j 
\textrm{ such that } d_{taxi}(p,q) \le  (1/2^j).
 \ee
 Thus 
  \be
 d_{GH}((X,d_{g_j}), (X, d_{taxi}) )  \le \delta_j=(1/2^j)(h_j + \sqrt{2}+1).
 \ee
and
 \be
 \diam(X, d_{g_j}) \le 2 +\delta_j.
 \ee
 where
 \be
 \delta_j=(1/2^j)(h_j + \sqrt{2}+1).
\ee
 
 So if we choose 
 \be \label{heights}
 h_j \to \infty \textrm{ and } h_j/2^j \to 0.
 \ee
 then we have a sequence of piecewise flat Riemannian manifolds
 $(X, g_j)$ such that 
 \be
\vol_{g_j}(X)\to \infty \textrm{ and } (X,d_{g_j}) \GHto (X, d_{taxi}). 
 \ee  

We need only verify that $d_{g_j}$ satisfies the H\"older bound (\ref{d_j}).   It
is easy to see that $g_h \ge g_0$ and so by the careful
rescaling done in the definition of $g_j$, we have
$g_j \ge g_0$, so $d_{g_j} \ge d_0=d_{g_0}$ satisfies
left side of (\ref{d_j}) with $\lambda=1$ and any $\beta\le1$.  

We claim that for any $\alpha \le 1$
we have a sequence $h_j$ satisfying (\ref{heights}) such that
\be\label{lambda-alpha}
d_{g_j}(p,q)\le \lambda_{\alpha} d_0(p,q)^\alpha.
\ee 
First observe that by (\ref{ongrid2}), (\ref{Haus1}),
(\ref{Haus2}), and $d_{taxi}(p,q)\le 2 d_{0}(p,q)$ we have
\be
d_{g_j}(p,q)\le 2 d_{0}(p,q) +\delta _j. 
\ee
We also know 
\be
\frac{g_j(v,v)}{g_0(v,v)}=\frac{g_{h_j}(v,v)}{g_0(v,v)}\le (2+h_j)^2
\ee
because no direction is stretched more than 2-fold or h-fold.
Thus
\be
d_{g_j}(p,q) \le \min\{ (2+h_j)d_0(p,q), 2 d_{0}(p,q) +\delta _j\}.
\ee
We need only show that for $s\in [0,2]$ we have
\be \label{thisthis}
f_j(s)=\min\{ (2+h_j)s, 2 s +\delta _j\}\le \lambda_{\alpha} s^\alpha.
\ee
It is easily seen to be true at $s=0$ and it holds at
$s=2$ if we take $\lambda_{\alpha}$ large enough that
\be
2\cdot 2 + \delta_j \le \lambda_{\alpha} 2^\alpha
\ee
which is easily done uniformly in $j$.  Since $f_j$ is
piecewise linear and $\lambda_{\alpha} s^\alpha$ is concave
down, we need only verify (\ref{thisthis}) holds at the
point where
\be
(2+h_j)s= 2 s +\delta _j
\ee
which is where $s=\delta_j/h_j$.   We need only show
\be
f_j(\delta_j/h_j)=\delta_j(2/h_j + 1) \le \lambda_{\alpha}(\delta_j/h_j)^\alpha
\ee
Since $h_j \to \infty$ this holds for large enough $ \lambda_{\alpha}$
if
\be
\delta_j \le \lambda_{\alpha}(\delta_j/h_j)^\alpha
\ee
So we need
\be
\delta_j^{1-\alpha}\le \lambda_{\alpha} h_j^{-\alpha} 
\ee
But $\delta_j=(1/2^j)(h_j + \sqrt{2}+1)$ so we need only show
\be
(1/2^j)^{1-\alpha} h_j^{1-\alpha}  \le \lambda_{\alpha} h_j^{-\alpha}. 
\ee
This works for any
\be
h_j \le \lambda_{\alpha}(1/2^j)^{\alpha-1}.
\ee
We can definitely choose such a sequence satisfying (\ref{heights})
so we are done proving the H\"older bound.

\bibliographystyle{alpha}
 \bibliography{AllenBryden}
\end{document}